\documentclass[11pt]{article}
\usepackage{amsthm, amsmath, amssymb, amsfonts, url, booktabs, tikz, setspace, fancyhdr, enumerate,mathtools}
\usepackage[margin = 1in]{geometry}

\usepackage[czech,english]{babel} 

\usepackage{hyperref}
\usepackage[shortlabels]{enumitem}
\usepackage{cleveref}

\definecolor{refkey}{gray}{.75}
\definecolor{labelkey}{gray}{.2}

\usepackage[textsize=scriptsize,colorinlistoftodos]{todonotes}

\newtheorem{theorem}{Theorem}[section]
\newtheorem{proposition}[theorem]{Proposition}
\newtheorem{lemma}[theorem]{Lemma}
\newtheorem{corollary}[theorem]{Corollary}

\newtheorem{question}[theorem]{Question}

\theoremstyle{definition}

\newtheorem{example}[theorem]{Example}

\theoremstyle{remark}
\newtheorem*{remark}{Remark}

\DeclareMathOperator{\im}{im}

\newcommand{\overbar}[1]{\mkern 1.5mu\overline{\mkern-1.5mu#1\mkern-1.5mu}\mkern 1.5mu}

\newcommand{\G}{\mathcal{G}}

\newcommand{\Hom}{\mathrm{Hom}}

\newcommand{\EE}{\mathbb{E}}

\newcommand{\RR}{\mathbb{R}}

\newcommand{\FF}{\mathcal{F}}

\newcommand{\C}{\mathcal{C}}

\newcommand{\e}{\varepsilon}

\newcommand{\Sol}{\mathrm{Sol}}
\newcommand{\ft}{\widehat}
\newcommand{\conj}{\overline}
\newcommand{\sub}{\mathrm{sub}}

\newcommand{\mockalph}[1]{}

\newcommand{\one}{\mathbf{1}}

\tikzstyle{p}+=[fill=black, circle, minimum width = 1pt, inner sep =
1pt]
\tikzstyle{w}+=[fill=white, draw, circle, minimum width = 1pt, inner sep =
1.5pt]

\title{On norming systems of linear equations}
\author{
Seokjoon Cho\thanks{Department of Mathematical Sciences, 
		Seoul National University, South Korea. 
		Email: {\tt znthlz@snu.ac.kr}.
}
\and
David Conlon\thanks{Department of Mathematics, California Institute of Technology, United States. Email: {\tt dconlon@caltech.edu}.
}
\and
Joonkyung Lee\thanks{Department of Mathematics, Yonsei University,  South Korea. Email: {\tt joonkyunglee@yonsei.ac.kr}.
}
\and
Jozef Skokan\thanks{Department of Mathematics, London School of Economics and Political Science, United Kingdom. Email: {\tt j.skokan@lse.ac.uk}.
}
\and
Leo Versteegen\thanks{Department of Mathematics, London School Economics and Political Science, United Kingdom. Email: {\tt lversteegen.math@gmail.com}.
}
}

\date{}

\begin{document}
\maketitle

\begin{abstract}
A system of linear equations $L$ is said to be norming if a natural functional $t_L(\cdot)$ giving a weighted count for the set of solutions to the system can be used to define a norm on the space of real-valued functions on $\mathbb{F}_q^n$ for every $n>0$. For example, Gowers uniformity norms arise in this way. In this paper, we initiate the systematic study of norming linear systems by proving a range of necessary and sufficient conditions for a system to be norming. Some highlights include an isomorphism theorem for the functional $t_L(\cdot)$, a proof that any norming system must be variable-transitive and the classification of all norming systems of rank at most two.
\end{abstract}

\section{Introduction}

One of the key tools in modern additive combinatorics, introduced by Gowers~\cite{G01} in his seminal work on Szemer\'edi's theorem, are the Gowers uniformity norms $\|\cdot\|_{U_k}$. For example, for an abelian group $G$ and a function $f : G \rightarrow \mathbb{R}$, the $U_2$-norm of $f$ is given by
\[\|f\|_{U_2}^4 = \EE\left[f(x_1) f(x_2) f(x_3) f(x_4)\right],\]
where the expectation is taken over all solutions to the linear equation $x_1-x_2 +x_3-x_4=0$. That this functional and its generalisations are indeed norms follows from first proving a certain H\"older-type inequality using iterated applications of the Cauchy--Schwarz inequality and this in turn implies the required triangle inequality. 

Our concern in this paper will be with the question of deciding which other systems of linear equations give rise to norms. Although many of our results do extend to other abelian groups, we will work throughout with groups $G$ of the form $\mathbb{F}_q^n$ where $q$ is a prime power and $n$ is a parameter that we often regard as tending to infinity. In this context, an \emph{$m\times k$ linear system on $G$} is a matrix $L\in \mathbb{F}_q^{m\times k}$ with linearly independent rows, though we use the term also for the corresponding system of linear equations. In line with this correspondence, by an \emph{equation} in $L$, we mean a vector in the row space of $L$, with, for instance, the $i$th row $(L_{i1}, \dots, L_{ik})$ of $L$ corresponding to the equation $L_{i1}x_1+\dots+L_{ik}x_k=0$.

Given a linear system $L$, we write $\Sol(L)=\{x\in G^k:Lx=0\}$ for the solution set of $L$ in $G^k$, i.e., those  $(x_1,\ldots,x_k)\in G^k$ which satisfy $L_{i1}x_1+\cdots+L_{ik}x_k=0$ for all $i\in [m]$. 
Then, for real-valued functions $f_1, \dots , f_k$ on $G$, we set
\begin{align*}
    t_L(f_1,\dots,f_k) 
    := \EE_{(x_1,\dots,x_k)\in\Sol(L)} f_1(x_1)\cdots f_k(x_k),
\end{align*}
where, here and throughout, $\EE_{(x_1,\dots,x_k)\in\Sol(L)}$ is a shorthand for $\frac{1}{|\Sol(L)|}\sum_{(x_1,\dots,x_k)\in\Sol(L)}$.
 
If we now set $t_L(f):=t_L(f,\dots,f)$ and $\|f\|_L:=|t_L(f)|^{1/k}$, we can say that 
$L$ is \emph{(semi-)norming} if $\|\cdot\|_L$ defines a (semi-)norm on the space of real-valued functions on $\mathbb{F}_q^n$ for all $n$. 
As it is easy to see that $\|cf\|_L = |c|\|f\|_L$, a linear system $L$ is semi-norming if and only if $\|\cdot\|_L$ satisfies the triangle inequality. We also say that $L$ is \emph{weakly (semi-)norming} if $\|f\|_{r(L)}:=t_L(|f|)^{1/k}$ defines a (semi-)norm.  
It will turn out that only certain degenerate systems are (weakly) semi-norming rather than norming, so we will largely ignore the distinction. As such, our main question can be stated in the following more concrete terms. 

\begin{question} \label{que:main}
Which linear systems $L$ are norming or weakly norming?
\end{question}

There is now a considerable body of work studying the analogue of \Cref{que:main} for graphs. In this setting, the question of classification was first raised by Lov\'asz~\cite{L08} and then studied in some depth by Hatami~\cite{H10}, who proved some necessary conditions, namely, that any weakly norming graph must be bipartite, balanced and biregular. He also showed that several families of graphs, including hypercubes, are weakly norming, a result which was greatly extended by Conlon and Lee~\cite{CL16}, who showed that every finite reflection group gives rise to a family of weakly norming graphs. They also conjectured a converse saying that all weakly norming graphs should arise from finite reflection groups in this way, though this remains wide open, with the best result towards it being a result of Sidorenko~\cite{Sid20} saying that weakly norming graphs must be edge-transitive. 

Here we will prove a collection of necessary and sufficient conditions for a linear system to be norming or weakly norming, many of which run parallel to those in the graph setting, but often requiring significantly different techniques to prove. 
For instance, we prove an analogue of Sidorenko's edge-transitivity result saying that any weakly norming linear system is variable-transitive, in the sense that deleting any given variable yields an isomorphic subsystem. The proof of this requires an isomorphism theorem for the functional $t_L(\cdot)$, saying that if $t_L(f) = t_M(f)$ for all non-negative real-valued functions $f$, then the systems $L$ and $M$ are isomorphic. For graphs, the analogous statement, due to Lov\'asz~\cite{Lov67}, is reasonably straightforward, while the arithmetic version requires significantly more work. On the other hand, the additional power given to us by the ability to apply Fourier methods means that classification, at least for low rank systems, comes within reach. In particular, we are able to classify all weakly norming systems of rank at most two. 

Since we prove many results across several sections, we spell out the main contributions of each section below:

\begin{itemize}

\item
In Section~\ref{sec:prelim}, we set out some preliminaries, first showing how the condition that $L$ be (weakly) semi-norming can be rephrased in terms of satisfying a H\"older-type inequality. We then introduce some basic discrete Fourier analysis and rephrase this H\"older-type inequality in Fourier terms. We also define and note some basic properties of
subsystems of linear systems.

\item
In Section~\ref{sec:necessary-conditions}, we establish some basic necessary conditions for a linear system to be weakly norming. For instance, we show that any weakly norming system must be translation invariant, so, in particular, each equation must have all coefficients summing to zero. We also show that any weakly norming system satisfies the arithmetic analogue of Sidorenko's conjecture, an important conjecture in graph theory (see, for example,~\cite{CFS10, CKLL15, CL20, Sz15} on the graph case and \cite{Alt23, FPZ19, KLM23} on the arithmetic case). It is also here that we say more about the distinction between (weakly) norming and semi-norming systems and why we can, for the most part, ignore the distinction.

\item
In Section~\ref{sec:iso}, we prove the aforementioned arithmetic analogue of the isomorphism theorem from graph theory and use it to prove that any weakly norming system must be variable-transitive. The main idea of the proof of the isomorphism theorem is to show that the condition that $t_L(f) = t_M(f)$ for all non-negative real-valued functions $f$ lifts to a similar conclusion where $f$ ranges over all complex-valued functions and then an appropriate choice of $f$ allows us to derive the required conclusion. Once this is in place, we may follow Sidorenko's technique to establish variable-transitivity.

\item
In Section~\ref{sec:rank2}, we classify all weakly norming linear systems of rank at most two. For single equations, any such equation must be Schatten, meaning that it is an equation of the form $a_1 x_1 + \dots + a_k x_k = 0$ with all $a_i \in \{a, -a\}$ and $\vert\{i\in [k]:a_i=a\}\vert=\vert\{i\in [k]:a_i=-a\}\vert$. For rank two systems, we need an additional operation, called subdivision, which takes a (weakly) norming system of linear equations and produces another (weakly) norming system with twice as many variables. The main result of this section then says that any rank two weakly norming linear system arises through subdividing some basic examples. 

\item
In Section~\ref{sec:forcing}, we show that weakly norming linear systems are forcing. The Sidorenko property, established for weakly norming linear systems in Section~\ref{sec:necessary-conditions}, says that, of all functions $f$ with a given density, the constant function is the one which minimises $t_L(f)$. The stronger forcing property, whose analogue for graphs is again well-studied, says that the constant function is also the \emph{unique} function which minimises $t_L(f)$. 

\item
In Section~\ref{sec:Cayley}, we show how each (weakly) norming hypergraph gives rise to a (weakly) norming linear system. This is a key source of examples for us, though the examples are given through a natural parametrisation of the solution set of a set of linear equations rather than in terms of the linear equations themselves. For systems arising from graphs, we show that a simple technique can be used to find the system of equations itself, though we also explain why this technique does not straightforwardly extend to systems arising from hypergraphs.

\item
In Section~\ref{sec:complex}, we prove the arithmetic analogue of an important result of Lee and Sidorenko \cite{LS21} saying that any norming graph is also norming for complex-valued functions. In order for this to work, one must take the conjugate of some of the terms in the expression for $t_L(f)$ and the key difficulty in the proof, which closely follows that in~\cite{LS21}, is in showing that there is some such assignment which works. 

\item
Finally, in Section~\ref{sec:conc}, we conclude with some brief further remarks on open problems.

\end{itemize}

\section{Preliminaries} \label{sec:prelim}

In this section, we collect some results which will be needed throughout the paper. 
We first discuss some reformulations of the norming property, beginning with the proposition below, an arithmetic analogue of a result of Hatami~\cite{H10} on graph norms, which was itself inspired by Gowers' proof~\cite{G01} that the uniformity norms are indeed norming. 
It says that the triangle inequality condition needed for a system to be norming is equivalent to a certain H\"older-type condition.

For the proof, we will need some notation. Given functions $f,g:\mathbb{F}_q^n\rightarrow \mathbb{C}$, their \emph{tensor product} $f\otimes g$ is the function on $\mathbb{F}_q^n\times \mathbb{F}_q^n$ defined by
\begin{align*}
    f\otimes g (y,z) = f(y)g(z).
\end{align*}
Using the natural isomorphism between  $(\mathbb{F}_q^n)^2$ and $\mathbb{F}_q^{2n}$, $f\otimes g $ can be seen as a function on $\mathbb{F}_q^{2n}$. Moreover, for $f_i,g_i:\mathbb{F}_q^n\rightarrow \mathbb{C}$,
\begin{align*}
    t_L (f_1\otimes g_1, f_2\otimes g_2,\dots,f_k\otimes g_k) = t_L (f_1,\dots,f_k)t_L(g_1,\dots,g_k).
\end{align*}
To see this, let $(x_1,\dots,x_k)\in (\mathbb{F}_q^{2n})^k$ and write each $x_i =(y_i,z_i)$ with $y_i,z_i\in \mathbb{F}_q^n$, so that $f\otimes g(x_i)=f(y_i)g(z_i)$. Then, since $(x_1,\dots,x_k)$ satisfies $L(x_1,\dots,x_k)=0$ if and only if $L(y_1,\dots,y_k)=L(z_1,\dots,z_k)=0$, we have 
\begin{align*}
    t_L (f_1\otimes g_1, f_2\otimes g_2,\dots,f_k\otimes g_k) &= \EE_{(x_1,\dots,x_k)\in \Sol(L)} f_1\otimes g_1(x_1)\cdots f_k\otimes g_k (x_k) \\
    &= \EE_{(y_1,\dots,y_k)\in \Sol(L)}  \EE_{(z_1,\dots,z_k)\in \Sol(L)} f_1(y_1) g_1(z_1)\cdots f_k(y_k)g_k (z_k)\\
    &=t_L (f_1,\dots,f_k)t_L(g_1,\dots,g_k).
\end{align*}
For a positive integer $m$, the $m$-th tensor power $f^{\otimes m}:(\mathbb{F}_q^n)^m \rightarrow \mathbb{C}$ is the function given by
\begin{align*}
    f^{\otimes m} = \underbrace{f\otimes f \otimes \cdots\otimes f}_{m \text{ times}}.
\end{align*}
In particular, by iterating the multiplicativity of tensor products, we obtain 
\begin{align*}
    t_L(f_1^{\otimes m}, \dots , f_k^{\otimes m}) =t_L(f_1,f_2,\dots,f_k)^m.
\end{align*}

\begin{proposition}\label{Holder}
A linear system $L$ in $k$ variables is semi-norming if and only if the inequality 
\begin{align}\label{eq:rainbow}
    |t_L(f_1,\dots,f_k)|\leq \prod_{i=1}^{k}\|f_i\|_L
\end{align}
holds for all $n \in \mathbb{N}$ and all $f_1,\dots,f_k:\mathbb{F}_q^n\rightarrow\mathbb{R}$.
\end{proposition}

\begin{proof}
Suppose the functional $\|\cdot\|_L$ is a seminorm on the space of
real-valued functions on $\mathbb{F}_q^n$ for all  $n\in\mathbb{N}$ and the inequality \eqref{eq:rainbow} does not hold for some functions $f_1, \dots , f_k$ on $G = \mathbb{F}_q^n$, i.e., $|t_L(f_1,f_2,\dots,f_k)|> \prod_{i=1}^{k}\|f_i\|_L$. By renormalising, we may assume that $|t_L(f_1,f_2,\dots,f_k)|=c>1$ and $\|f_i\|_L\leq 1$ for each $i=1,\dots,k$. Then, for any even integer $m$,
\begin{align*}
    \|f_1^{\otimes m}+\cdots + f_k^{\otimes m}\|^k_L &= |t_L(f_1^{\otimes m}+\cdots + f_k^{\otimes m})|= \left|\sum_{1\leq i_1, \dots, i_k \leq k} t_L(f_{i_1}^{\otimes m},\dots, f_{i_k}^{\otimes m})\right|\\
    &=\sum_{1\leq i_1, \dots, i_k \leq k} t_L(f_{i_1},\dots,f_{i_k})^m\geq t_L(f_1,\dots,f_k)^m = c^m.
\end{align*}
On the other hand, $\|f_1^{\otimes m}\|_L+\cdots+\|f_k^{\otimes m}\|_L=\|f_1\|_L^m +\cdots +\|f_k\|_L^m \leq k$. Hence, if $m$ is sufficiently large, we have 
\begin{align*}
    \|f_1^{\otimes m}\|_L+\cdots+\|f_k^{\otimes m}\|_L \leq k< c^{m/k} \leq \|f_1^{\otimes m}+\cdots + f_k^{\otimes m}\|_L,
\end{align*}
which contradicts the triangle inequality for $\|\cdot \|_L$ on $\mathbb{F}_q^{nm}$.

Conversely, suppose that the inequality \eqref{eq:rainbow} holds. Then, for each $f_1$ and~$f_2$,
\begin{align*}
    \|f_1 +f_2\|_L^k &= |t_L(f_1+f_2)|\leq \sum_{1 \le i_1,\dots,i_k \le 2}|t_L(f_{i_1},\dots,f_{i_k})|\\
    &\leq \sum_{1 \le i_1,\dots,i_k \le 2} \|f_{i_1}\|_L \cdots \|f_{i_k}\|_L
    =(\|f_1\|_L+\|f_2\|_L)^k.
\end{align*}
Thus, $\|\cdot \|_L$ satisfies the triangle inequality.
\end{proof}

A similar argument also works in the weakly norming case, yielding the following result, whose proof we omit.

\begin{proposition}\label{weakly_Holder}
A linear system $L$ is weakly semi-norming if and only if the inequality
\begin{align*} 
    t_L(f_1,\dots,f_k)\leq \prod_{i=1}^{k}\|f_i\|_L =  \prod_{i=1}^{k}\|f_i\|_{r(L)}
\end{align*}
holds for all $n \in \mathbb{N}$ and all non-negative $f_1,\dots,f_k:\mathbb{F}_q^n\rightarrow\mathbb{R}_{\geq 0}$.
\end{proposition}

We will make considerable use of discrete Fourier analysis on the finite abelian groups $\mathbb{F}_q^n$. We give a brief refresher. Setting $G = \mathbb{F}_q^n$, we write $\ft{G}$ for the dual group of $G$, i.e., the group of homomorphisms from $G$ to the multiplicative group $\mathbb{C}$ of complex numbers, which is easily seen to be isomorphic to $G$ itself. 

 Denoting the inner product of $\xi, x\in G$ by $\xi^Tx$, the Fourier transform $\ft{f}:\ft{G}\rightarrow\mathbb{C}$ of a function $f:G\rightarrow\mathbb{C}$ is defined by 
 $$\ft{f}(\xi):=\EE_{x\in G}f(x)e(-\xi^T x),$$ 
 where $e(y) := e^{2\pi i y/q}$, and we have the Fourier inversion formula $$f(x)=\sum_{\xi \in \ft{G}}\ft{f}(\xi)e(\xi^T x).$$ 
Using these formulas, it is easily checked that $f$ is real-valued if and only if $\ft{f}(\xi)$ and $\ft{f}(-\xi)$ are complex conjugates for every $\xi\in \ft{G}$.

We write $f_1*\cdots*f_k$ for the \emph{convolution} $f_1*\cdots*f_k(x):=\EE_{y_1+\cdots+y_k=x}f_1(y_1)\cdots f_k(y_k)$. 
In particular, if both $f_1$ and $f_2$ are always non-negative, then so is $f_1*f_2$. 
The importance of convolution lies in the fact that the Fourier transform $(f_1*\cdots* f_k)^{\wedge}$ of $f_1*\cdots* f_k$ satisfies $(f_1*\cdots* f_k)^{\wedge} = \ft{f_1}\cdots\ft{f_k}$.

Let $L$ be a $1\times k$ linear system such that no entry of $L$ is zero. For $g_j(y):=f_j(L_{1j}^{-1} y)$, we have
\begin{align*}
    t_L (f_1,\dots,f_k) = \EE_{\Sol(L)} f_1(x_1)\cdots f_k(x_k) 
    = g_1 *g_2 *\cdots *g_k (0) = \sum_{\xi\in\ft{G}}\prod_{j=1}^k \ft{g_j}(\xi)
    =\sum_{\xi\in\ft{G}}\prod_{j=1}^k \ft{f_j}(L_{1j}\xi).
\end{align*}

We will also make use of the analogous Fourier inversion formula for linear systems with more than one equation.

\begin{proposition}\label{prop:fourier}
Let $L\in \mathbb{F}_q^{m\times k}$ be a linear system with coefficients $(L_{ij})_{i\in [m],j\in [k]}$. Then 
\begin{align}\label{eq:inversion}
    t_L(f_1, \dots, f_k) = \sum_{(\xi_1, \dots, \xi_m) \in \ft{G}^m} \prod_{j=1}^{k}\ft{f_j}\left(\sum_{i=1}^m L_{ij}\xi_i\right)
\end{align}
for any $f_1, \dots, f_k:\mathbb{F}_q^n\rightarrow\mathbb{R}$.
\end{proposition}
\begin{proof}
    Expanding out the right-hand side of \eqref{eq:inversion}, we get
    \begin{align}\label{eq:inversion-expansion}
        \sum_{\xi \in \ft{G}^m} \prod_{j=1}^{k}\ft{f_j}\left(\sum_{i=1}^m L_{ij}\xi_i\right)
        &=\sum_{\xi\in \ft{G}^m} \prod_{j=1}^{k} \left( \EE_{x_j\in G} f_j(x_j)e\left(-\sum_{i=1}^m L_{ij}\xi_i^T x_j\right)
        \right)\nonumber\\
        &=\EE_{x\in G^k}\sum_{\xi\in \ft{G}^m} e\left(-\sum_{i=1}^m \xi_i^T \sum_{j=1}^k L_{ij} x_j\right)\prod_{j=1}^{k} f_j(x_j)\nonumber\\
        &=\frac{1}{\vert G\vert^k}\sum_{x\in G^k} \prod_{i=1}^m \left(\sum_{\xi_i\in \ft{G}} e\left(-\xi_i^T \sum_{j=1}^k L_{ij} x_j\right)\right)\prod_{j=1}^{k} f_j(x_j).
    \end{align}
    For fixed $x\in G^k$, we have
    \begin{align*}
        \sum_{\xi_i\in \ft{G}} e\left(-\xi_i^T \sum_{j=1}^k L_{ij} x_j\right)=\begin{cases}
            \vert G\vert &\text{if } \sum_{j=1}^k L_{ij} x_j=0,\\
            0 &\text{otherwise.}
        \end{cases}
    \end{align*}
    Inserting this into \eqref{eq:inversion-expansion}, we obtain
    \begin{align*}
        \sum_{\xi \in \ft{G}^m} \prod_{j=1}^{k}\ft{f_j}\left(\sum_{i=1}^m L_{ij}\xi_i\right)=\frac{\vert G\vert^m}{\vert G\vert^k} \sum_{x\in \Sol(L)} \prod_{j=1}^{k} f_j(x_j)=t_L(f_1,\ldots,f_k),
    \end{align*}
    as desired.
\end{proof}

Thus, by~\Cref{Holder}, the linear system $L$ is semi-norming if and only if 
\begin{align*}
    \left\vert \sum_{\vec{\xi} \in \ft{G}^m} \ft{f_1}\left(\sum_{i=1}^m L_{i1}\xi_i\right)\cdots\ft{f_k}\left(\sum_{i=1}^m L_{ik}\xi_i\right) \right\vert^k \leq 
    \prod_{j=1}^k \left\vert \sum_{\vec{\xi} \in \ft{G}^m} \ft{f_j}\left(\sum_{i=1}^m L_{i1}\xi_i\right)\cdots \ft{f_j}\left(\sum_{i=1}^m L_{ik}\xi_i\right)    \right\vert
\end{align*}
or, equivalently, writing $L^t_j$ for the $j^\text{th}$ column of $L$ as a row vector, 
\begin{align}\label{eq:fourier_holder}
    \left\vert \sum_{\vec{\xi} \in \ft{G}^m} \ft{f_1}\left(L_1^t\vec{\xi}\right)\cdots\ft{f_k}\left(L_k^t\vec{\xi}\right) \right\vert^k \leq 
    \prod_{j=1}^k \left\vert \sum_{\vec{\xi} \in \ft{G}^m} \ft{f_j}\left(L_1^t\vec{\xi}\right)\cdots \ft{f_j}\left(L_k^t\vec{\xi}\right)    \right\vert
\end{align}
for any $f_1, \dots, f_k:\mathbb{F}_q^n\rightarrow\mathbb{R}$.

\begin{example}\label{ex:K23}
    Let $L$ be the $2\times 6$ system defined by the matrix
    \begin{align*}
        \begin{pmatrix*}[r]
            1 & -1 & -1 & 1 & 0 & 0 \\
            0 & 0 & 1 & -1 & -1 & 1 \\
        \end{pmatrix*}.
    \end{align*}
    In other words, $\Sol(L)$ consists of solutions to the system of linear equations $x_1-x_2=x_3-x_4=x_5-x_6$. Then, by~\eqref{eq:inversion},
    \begin{align*}
        t_L(f_1,\ldots,f_6) = \sum_{(\xi_1, \xi_2) \in \ft{G}^2} \ft{f_1}(\xi_1)\ft{f_2}(-\xi_1)\ft{f_3}(-\xi_1+\xi_2)\ft{f_4}(\xi_1-\xi_2)\ft{f_5}(-\xi_2)\ft{f_6}(\xi_2).
    \end{align*}
    In particular, 
    \begin{align*}
        t_L(f) = \sum_{(\xi_1, \xi_2) \in \ft{G}^2} |\ft{f}(\xi_1)|^2|\ft{f}(\xi_1-\xi_2)|^2|\ft{f}(\xi_2)|^2 .
    \end{align*}
    We will see that this system $L$ is weakly norming in~\Cref{ex:theta}.
\end{example}

\medskip

We now define subsystems of linear systems, which bear some resemblance to subgraphs of graphs.
Let $L$ be an $m\times k$ system, let $i\in [k]$ and let $a^{(1)},\ldots,a^{(r)}$ be a basis for those vectors in the row space of $L$ for which the $i^\text{th}$ coordinate is zero. By \emph{deleting a variable} $x_i$, $i\in [k]$, we mean that we pass to the $r\times (k-1)$ system $L'$ with rows $\tilde{a}^{(1)},\ldots,\tilde{a}^{(r)}$ where $\tilde{a}^{(j)}$ is equal to $a^{(j)}$ with the $i^{\text{th}}$ coordinate removed. Note that by choosing a different basis $a^{(1)},\ldots,a^{(r)}$, we might obtain a different system $L'$. However, the solution space $\Sol(L')$ is 
always the projection of $\Sol(L)$ onto the set of coordinates $[k]\setminus\{i\}$, independently 
of the choice of basis.
By \emph{deleting an equation}, we mean removing the corresponding row vector from the basis for the row space of $L$, thereby reducing the dimension of the row space.
A \emph{subsystem} $L'$ of an $m\times k$ linear system $L$ is an $m'\times k'$ system with $k'\leq k$ and $m'\leq m$, obtained by deleting variables and equations from $L$. 
If a subsystem is obtained by deleting all variables indexed by $i\notin I$, then we say that it is \emph{induced} on $I\subseteq [k]$. 

\begin{example}\label{ex:deletion}
    Let $L$ be the $4\times 5$ system defined by the matrix
    \begin{align*}
        \begin{pmatrix*}[r]
        1 &-1 & 0 &0 &0 \\
        1 & 1 & -2 &0 &0 \\
        0 & 0& -2 &1 & 1 \\
        0 & 0 &0 &1 & -1 \\
        \end{pmatrix*}.
    \end{align*}
    Then $\Sol(L)$ is 1-dimensional, spanned by $(1,1,1,1,1)$. Deleting the variable $x_3$ from $L$ yields the subsystem $L'$ induced on $I=\{1,2,4,5\}$ whose matrix can be written as
    \begin{align*}
        \begin{pmatrix*}[r]
        1 &-1 & 0 &0 \\
        0 & 0 &1 & -1 \\
        1 & 1 &-1 &-1 \\
        \end{pmatrix*},
    \end{align*}
    where the last row is obtained by deleting the third entry from $(1,1,0,-1,-1)$, a vector in the row space of $L$.
    Thus, $\Sol(L)$ is again 1-dimensional, spanned by $(1,1,1,1)$.
\end{example}

If $L'$ is an induced subsystem of $L$, we may write $t_{L'}(f)$ as $t_L(f_1,\ldots,f_k)$ for some appropriate $f_1, \dots, f_k$. For graphs, this is obvious, but we have to be a little more careful in the arithmetic setting, so we include a proof.

\begin{proposition}\label{prop:subsystem}
    Let $L'$ be an $m'\times k'$ subsystem of an $m\times k$ linear system $L$, induced on $I\subseteq [k]$. Given $f:G\rightarrow\mathbb{C}$, let $f_i=f$ for $i\in I$ and $f_j=1$ for $j\notin I$. Then $t_{L'}(f)=t_L(f_1,\ldots,f_k)$.
\end{proposition}

\begin{proof}
    We claim that if $L'$ is induced on $I=[k-1]$, then, for any complex-valued functions $f_1,\ldots,f_{k-1}$ on $G$,
    \begin{align}\label{eq:subsystem} 
        t_L(f_1,\ldots,f_{k-1},1)=t_{L'}(f_1,\ldots,f_{k-1}).
    \end{align}
    The proof then follows from repeatedly applying this claim to reduce the index set while possibly relabelling the variables. 
    
    Suppose first that there are no equations in $\Sol(L)$ that involve $x_k$. Then $x_k$ is a `free' variable, so that $t_L(f_1,\ldots,g) = t_{L'}(f_1,\ldots,f_{k-1})\EE[g]$ for every function $g$. For $g\equiv 1$, this yields~\eqref{eq:subsystem}. We may therefore assume that there is at least one equation in $\Sol(L)$ that involves $x_k$.
    
    Let $p_k:G^k\rightarrow G^{k-1}$ be the projection map $(x_1,\ldots,x_k)\mapsto(x_1,\ldots,x_{k-1})$. Then
    \begin{align}\label{eq:projection}
        \nonumber t_L(f_1,\ldots,f_{k-1},1) &=\frac{1}{|\Sol(L)|}\sum_{(x_1,\dots,x_k)\in\Sol(L)} f_1(x_1)\cdots f_{k-1}(x_{k-1})\\ &= \frac{1}{|\Sol(L)|}\sum_{(x_1,\dots,x_{k-1})\in p_k(\Sol(L))} f_1(x_1)\cdots f_{k-1}(x_{k-1})|p_k^{-1}(x_1,\ldots,x_{k-1})|,
    \end{align}
    where $p_k^{-1}(\cdot)$ denotes the inverse image of $p_k$.
    
    Since at least one equation of $L$ involves $x_k$, $p_k$ is a bijection between $\Sol(L)$ and $\Sol(L')$.  
    Indeed, under $p_k$, each $(x_1,\ldots,x_k)$ projects to $(x_1,\ldots,x_{k-1})$, which satisfies all equations in $L'$, proving that $p_k(\Sol(L))\subseteq\Sol(L')$. Conversely, let $(x_1,\ldots,x_{k-1})\in\Sol(L')$. It is enough to check that this uniquely extends to $(x_1,\ldots,x_k) \in \Sol(L)$.
    Given an equation $a_1x_1+\cdots+a_kx_k=0$ in $L$ with $a_k\neq 0$, $(x_1,\ldots,x_{k-1})$ uniquely determines 
    \begin{align*}
        x_k = -\frac{a_1}{a_k}x_1 -\cdots -\frac{a_{k-1}}{a_k}x_{k-1}.
    \end{align*}
    Moreover, this choice of $x_k$ is consistent with the one obtained by using any other equation $b_1x_1+\cdots + b_kx_k=0$ with $b_k\neq 0$ since $L'$ includes the equation 
    \begin{align*}
         \frac{a_1}{a_k}x_1 +\cdots +\frac{a_{k-1}}{a_k}x_{k-1} =  \frac{b_1}{b_k}x_1 +\cdots +\frac{b_{k-1}}{b_k}x_{k-1}.
    \end{align*}
    Therefore, $p_k(\Sol(L))=\Sol(L')$ and $p_k$ is a bijection from $\Sol(L)$ to $\Sol(L')$, which turns~\eqref{eq:projection} into the identity~\eqref{eq:subsystem}.
\end{proof}

\section{Basic necessary conditions} \label{sec:necessary-conditions}

A $k\times k$ linear system $L$ with non-zero determinant satisfies $\Sol(L)=\{0\}$, which, for our purposes, may be seen as a degenerate case. 
As such, we shall assume $\Sol(L)\neq\{0\}$ in what follows, unless we are specifically speaking about degenerate systems. 
The following proposition asserts that all other weakly semi-norming systems are translation invariant, in the sense that if $(x_1,\dots,x_k)\in \Sol(L)$ and $g\in G$, then $(x_1+g,\dots,x_k+g)\in \Sol(L)$ as well. 
In particular, this means that the coefficients of each equation in a weakly norming system must add to $0$. 

\begin{proposition}\label{prop:degeneracy}
    Every weakly semi-norming $m\times k$ system $L$ is translation invariant. In particular, for all $i\in [k]$ and $v\in G$, 
    \begin{align*}
        \vert \{(x_1,\dots,x_k)\in \Sol(L): x_i=v\}\vert = \vert \Sol(L)\vert /\vert G\vert.
    \end{align*}
\end{proposition}

\begin{proof}
    Let $(y_1,\dots,y_k)\in \mathbb{F}_q^k$ be a non-zero element in $\Sol(L)$ with $y_j=v\neq 0$. 
    For $f_j=\mathbf{1}_{\{v\}}$ and $f_i=1$ for all $i\neq j$, we have
    \begin{align*}
        t_L(f_1,\dots,f_k) = \frac{|\{(x_1,\dots,x_k)\in\Sol(L):x_j=v\}|}{|\Sol(L)|},
    \end{align*}
    which is strictly positive. 
    By~\Cref{weakly_Holder}, it follows that $t_L(f_j)>0$.
    That is, $\Sol(L)$ contains $(v,v,\dots,v)$ and, hence, it also contains $(1,1,\dots,1)$. Thus, $L$ must be translation invariant.
\end{proof}

By this proposition, if $L$ is a non-degenerate weakly norming system, then each $(v, v, \dots, v) \in \Sol(L)$, so we have that 
$$\|f\|_{r(L)} \ge \frac{1}{|\Sol(L)|} \sum_v |f(v)|^k > 0$$ 
for all $f \neq 0$, implying that the system is in fact weakly norming. 

Another consequence of this proposition is that $\EE[f]^k \leq t_L(f)$
whenever $L$ is an $m \times k$ weakly norming system and~$f$ is a non-negative function on $G$. 
Indeed, let $f_1=f$ and $f_2=\cdots=f_k=1$. Then
\begin{align*}
    t_L(f_1,\dots,f_k) = \EE_{(x_1,\dots,x_k)\in\Sol(L)} f(x_1) = \EE_{x\in G} f(x)
\end{align*}
and, therefore, \Cref{weakly_Holder} implies that $\EE[f]^k \leq t_L(f)$. 
As mentioned in the introduction, linear systems which satisfy this inequality for all non-negative $f$ are called \emph{Sidorenko}.

\begin{corollary}\label{Sidorenko}
    Every weakly norming system $L$ is Sidorenko.
\end{corollary}

We note in passing that, since Sidorenko systems are necessarily translation invariant (see, for instance,~\cite{KLM23}), this corollary again implies that weakly norming systems are translation invariant.

By using~\Cref{prop:subsystem}, we can also generalise Corollary~\ref{Sidorenko} to say that every weakly norming system dominates all of its subsystems in the following sense (see~\cite{CL23} for more on the analogous property for graphs).

\begin{corollary}\label{cor:domination}
    Let $L'$ be a subsystem of an $m\times k$ weakly norming system $L$ induced on $I\subseteq[k]$ with $|I|=k'$. Then, for every non-negative function $f:G\rightarrow \mathbb{R}_{\geq 0}$,
    \begin{align*}
        t_L(f)^{1/k} \geq t_{L'}(f)^{1/k'}.
    \end{align*}
    Furthermore, if $L$ is norming, then $|t_L(f)|^{1/k}\geq |t_{L'}(f)|^{1/k'}$ for all $f:G\rightarrow\mathbb{R}$.
\end{corollary}

The \emph{girth} of a linear system $L$ is the minimum size of the support of a non-zero vector in the row space of $L$. We will now show that the girth of any Sidorenko system, and hence any weakly norming system, is even. This result was previously proved by Kam\v cev, Liebenau and Morrison~\cite{KLM23}, though we give a somewhat shorter proof using Fourier methods.

\begin{proposition}\label{prop:even-girth}
    The girth of a Sidorenko system is even.
\end{proposition}

\begin{proof}
    Let $L$ be an $m\times k$ system whose girth $\ell$ is odd. Suppose, without loss of generality, that one of the non-zero vectors in the row space of $L$ with minimum support has the form $(a_1,\ldots,a_\ell,0,\ldots,0)$ for non-zero $a_1,\ldots,a_\ell\in \mathbb{F}_q$. Let $\alpha\in (0,1)$ and $\gamma$ be a non-zero element of $\ft{G}$ and consider the function $f_\alpha\colon G\rightarrow \RR$ with Fourier transform 
    \begin{equation*}
        \ft{f_\alpha}(\xi) = \begin{cases}
                    1&\text{if $\xi=0$},\\
                    -\alpha&\text{if $\xi=\pm a_i\gamma$},\\
                    0&\text{otherwise},\\
                    \end{cases}
    \end{equation*}
    noting that if $\alpha\leq 1/2\ell$, then $f_\alpha$ takes values in $[0,2]$. By \Cref{prop:fourier}, we have 
    \begin{equation*}
        t_L(f_\alpha)=\sum_{\xi \in \hat{G}^m}\prod_{j=1}^k \ft{f_\alpha}\left(L^t_j\xi\right)\leq 1+(-\alpha)^\ell+O(\alpha^{\ell+1}).
    \end{equation*}
    If $\alpha$ is sufficiently small, then this is less than 1, but, because $\EE[f_\alpha]=\ft{f_\alpha}(0)=1$, this means that $L$ is not Sidorenko.
\end{proof}

Recall that if $L$ is an $m\times k$ norming system, then $\|f\|_L$ equals $|t_L(f)|^{1/k}$. The following proposition shows that we do not need the modulus.

\begin{proposition}\label{prop:positivity}
    If $L$ is a norming system, then $t_L(f)> 0$ for all $f\neq 0$. In particular, $L$ must have an even number of variables.
\end{proposition}

\begin{proof}
    Suppose, for the sake of contradiction, that there is some non-zero $f$ for which $t_L(f)<0$. Let $v_1$ and $v_2$ be two distinct elements of $G$. By the translation invariance of $L$, $t_L(\one_{\{v_1\}})$ and $t_L(\one_{\{v_2\}})$ are both positive. Hence, by the continuity of $t_L$ and the intermediate value theorem, there are $c_1,c_2\in \RR_{> 0}$ such that $t_L(f+c_1\one_{\{v_1\}})=t_L(f+c_2\one_{\{v_2\}})=0$. However, since $L$ is norming, this implies that $f+c_1\one_{\{v_1\}}=f+c_2\one_{\{v_2\}}=0$, which cannot be the case as $\one_{\{v_1\}}$ and $\one_{\{v_2\}}$ are clearly linearly independent.
\end{proof}

Finally, we show that the only systems which are semi-norming but not norming are the zero matrices.

\begin{proposition}\label{semi-norming-trivial}
    If a system $L$ is semi-norming but not norming, then it must be the zero matrix.
\end{proposition}

\begin{proof}
    There must exist $i\in [k]$ such that if $x\in \Sol(L)$ satisfies $x_j=0$ for all $j\neq i$, then $x_i=0$ too. Indeed, otherwise $L$ is the zero matrix and we are done. By the translation invariance of $L$, the same is true if we replace 0 by any other $v\in G$, i.e., if $x_j=v$ for all $j\neq i$, then $x_i=v$ too. Without loss of generality, we may assume that $i=1$. Suppose now that $f$ is a function with $\Vert f\Vert_L=0$. By \Cref{Holder}, we have that
    \begin{align*}
        |t_L(f,\one_{\{v\}},\ldots,\one_{\{v\}})|\leq \Vert f\Vert_L/|\Sol(L)|^{k-1} = 0
    \end{align*}
    for all $v$. On the other hand, we know that the only solution that contributes to $t_L(f,\one_{\{v\}},\ldots,\one_{\{v\}})$ is $(v,\ldots,v)$, meaning that $|f(v)|=|t_L(f,\one_{\{v\}},\ldots,\one_{\{v\}})|\vert \Sol(L)\vert=0$. Since this holds for all $v$, we have $f=0$.
\end{proof}

\section{The isomorphism theorem} \label{sec:iso}

For graphs, it is a fundamental fact (see, for example,~\cite[Theorem 5.29]{L12}) that if $|\Hom(H_1,F)|=|\Hom(H_2,F)|$ for every graph $F$, then $H_1$ and $H_2$ must be isomorphic. In this section, we prove the arithmetic analogue of this (left-)isomorphism theorem and apply it to give some further necessary conditions for a linear system to be weakly norming.

To state the result, we should first define what isomorphism means for linear systems: two $m\times k$ linear systems $L$ and $M$ are \emph{isomorphic} if $M$ can be obtained by applying row operations and column permutations to $L$. That is, $L$ and $M$ are isomorphic if the vector space $\Sol(M)$ can be obtained from $\Sol(L)$ by simply permuting the indices of the variables. 

\begin{theorem}\label{thm:isom}
    Two $m\times k$ linear systems $L$ and $M$ are isomorphic if and only if $t_L(f)=t_M(f)$ for all positive integers $n$ and all non-negative functions $f:\mathbb{F}_q^n\rightarrow\mathbb{R}_{\geq 0}$.
\end{theorem}

There are two main steps in the proof. The first is encapsulated in the following lemma, which gives us much more flexibility in substituting various functions $f$ into $t_L(f)$.
We define the \emph{symmetrised functional} $\tau_L(\cdot)$ by
\begin{align*}
    \tau_L(f_1,\dots,f_k) := \sum_{\pi\in S_k} t_L(f_{\pi(1)},\dots,f_{\pi(k)})
\end{align*}
for any functions $f_1,\dots,f_k : \mathbb{F}_q^n \rightarrow \mathbb{C}$, where $S_k$ denotes the set of all permutations of $[k]$.

\begin{lemma}\label{lem:boost}
     Let $L$ and $M$ be two $m\times k$ linear systems. Then the following are equivalent:
     \begin{enumerate}[(i)]
         \item $t_L(f)=t_{M}(f)$ for every non-negative function $f:\mathbb{F}_q^n\rightarrow\mathbb{R}_{\geq 0}$;\label{it:+}
         \item $t_L(f)=t_{M}(f)$ for every real-valued function $f:\mathbb{F}_q^n\rightarrow\mathbb{R}$;\label{it:real}
         \item $\tau_L(f_1,\dots,f_k)=\tau_{M}(f_1,\dots,f_k)$ for any complex-valued functions $f_1,\dots,f_k$ on $\mathbb{F}_q^n$.\label{it:complex}
     \end{enumerate}
\end{lemma}

\begin{proof}
    Since the implication \ref{it:complex}$\Rightarrow$\ref{it:+} is trivial, it will suffice to show that \ref{it:+}$\Rightarrow$\ref{it:real} and \ref{it:real}$\Rightarrow$\ref{it:complex}.
    We first show that  \ref{it:+}$\Rightarrow$\ref{it:real}. Let $\varepsilon_1,\dots,\varepsilon_k$ be i.i.d.~uniform random variables taking values in $\{\pm 1\}$ and set  $\vec{\varepsilon}=(\varepsilon_1,\dots,\varepsilon_k)$.
    For a real-valued function $f$ on $\mathbb{F}_q^n$, 
    \begin{align*}
        \EE_{\vec{\varepsilon}} \left[\varepsilon_1\varepsilon_2\cdots\varepsilon_k t_L\left(\sum_{i=1}^k(|f|+\varepsilon_if)\right) \right]= k! t_L(f).
    \end{align*}
    Indeed, expanding $t_L\left(\sum_{i=1}^k(|f|+\varepsilon_if)\right)$ yields 
    a sum of terms of the form $t_L(h_1,h_2,\dots,h_k)$ where each $h_i$ is either $|f|$ or $\varepsilon_j f$ for some $j$. Using $\EE [\varepsilon_j]=0$ and $\EE [\varepsilon_j^2] =1$, we see that the only terms that do not vanish after averaging over $\vec{\varepsilon}\in\{\pm 1\}^k$ are those $t_L(h_1,h_2,\dots,h_k)$ such that each $h_i=\varepsilon_j f$ for a unique $j$. This is exactly what the right-hand side represents. As $\sum_{i=1}^k(|f|+\varepsilon_if)\geq 0$ for any real-valued $f$, we have that \ref{it:+}$\Rightarrow$\ref{it:real}, as required.

    To show that \ref{it:real}$\Rightarrow$\ref{it:complex}, let $z_1,\dots,z_k$ be i.i.d.~uniform random unit complex numbers and set $\vec{z}=(z_1,\dots,z_k)$.
    For complex-valued functions $f_1,\dots,f_k$ on $\mathbb{F}_q^n$, 
    \begin{align*}
        \EE_{\vec{z}} \left[z_1z_2\cdots z_k t_L\left(\sum_{i=1}^k(z_i \overbar{f_i} +\overbar{z_i}f_i)\right) \right]= \sum_{\pi\in S_k}t_L(f_{\pi(1)},\dots,f_{\pi(k)}).
    \end{align*}
    Indeed, we can use that $\EE [z_i] = \EE[z_i^2] = 0$ and $\EE[z_i\overbar{z_i}]=1$ to deduce that all terms except those of the form $t_L(f_{\pi(1)},\dots,f_{\pi(k)})$ vanish in the expansion of $t_L\left(\sum_{i=1}^k(z_i \overbar{f_i} +\overbar{z_i}f_i)\right)$ after averaging over $\vec{z}\in (S^1)^k$. As $\sum_{i=1}^k(z_i \overbar{f_i} +\overbar{z_i}f_i)$ is always real-valued,  
    we have that \ref{it:real}$\Rightarrow$\ref{it:complex}, as required.
\end{proof}

The second step involves substituting suitable functions $\ft{g_1},\ft{g_2},\dots,\ft{g_k}$ into the inversion formula~\eqref{eq:inversion} to distinguish $t_L(\cdot)$ and $t_{M}(\cdot)$ if $L$ and $M$ are not isomorphic.
We first record a technical lemma that will be needed in our computation.  

\begin{lemma}\label{lem:rowspace}
    Let $L$ and $M$ be $m\times k$ linear systems and let $\gamma\in \ft{G}^m$.
For each $j=1,2,\dots,k$, let $f_j^{(\gamma)}$ be the function with Fourier transform $\ft{f_j^{(\gamma)}}(\eta):=\mathbf{1}\big[\eta = L_j^t\gamma\big]$. Then 
\begin{align*}
    t_{M}(f_1^{(\gamma)}, \ldots, f_k^{(\gamma)})= \sum_{\xi \in \ft{G}^m}\mathbf{1}\big[M^t\xi=L^t\gamma\big].
\end{align*}
\end{lemma}

\begin{proof}
    By~\eqref{eq:inversion}, it follows that 
\begin{align*}
    t_{M}(f_{1}^{(\gamma)}, \dots, f_{k}^{(\gamma)}) = \sum_{\xi\in \ft{G}^m} \prod_{j=1}^{k}\ft{f_{j}^{(\gamma)}}\left(M_{j}^t\xi \right)= \sum_{\xi \in \ft{G}^m} \prod_{j=1}^{k} \mathbf{1}\big[M_{j}^t\xi = L_{j}^t\gamma\big],
\end{align*}
which proves the desired identity.
\end{proof}

\begin{proposition}\label{lem:distinguish}
    Let $L$ and $M$ be two non-isomorphic $m\times k$ linear systems. Then there exist $n\in\mathbb{N}$ and complex-valued functions $g_1,\dots,g_k$ on $\mathbb{F}_q^n$ such that $\tau_L(g_1,\dots,g_k)\neq \tau_M(g_1,\dots,g_k)$.
\end{proposition}
\begin{proof}
Suppose $G = \ft{G} = \mathbb{F}_q^n$ and, for each $\gamma\in \ft{G}^m$, let $f_1^{(\gamma)}, \ldots, f_k^{(\gamma)}$ be defined as in~\Cref{lem:rowspace}. Then, for $\pi\in S_k$,
\begin{align*}
    t_{M}(f_{\pi(1)}^{(\gamma)}, \dots, f_{\pi(k)}^{(\gamma)}) = \sum_{\xi \in \ft{G}^m}\mathbf{1}\big[M^t\xi =\pi(L)^t\gamma\big],
\end{align*}
where $\pi(L)$ denotes the linear system obtained by permuting the $k$ columns of $L$ under $\pi$. 
For $\Gamma=(\gamma_1,\dots,\gamma_\ell)\in(\ft{G}^m)^\ell$,
let $g_{i}^{(\Gamma)} = f_i^{(\gamma_1)}\otimes f_i^{(\gamma_2)} \otimes \cdots \otimes f_i^{(\gamma_\ell)}$, where $\ell$ will be chosen later. 
Then
\begin{align*}
    t_M(g_1^{(\Gamma)},g_2^{(\Gamma)},\dots,g_k^{(\Gamma)})= t_M(f_1^{(\gamma_1)},\dots,f_k^{(\gamma_1)})t_M(f_1^{(\gamma_2)},\dots,f_k^{(\gamma_2)})\cdots t_M(f_1^{(\gamma_\ell)},\dots,f_k^{(\gamma_\ell)})
\end{align*}
and, therefore,
\begin{align*}
    \sum_{\Gamma\in(\ft{G}^m)^{\ell}}t_M(g_1^{(\Gamma)},g_2^{(\Gamma)},\dots,g_k^{(\Gamma)})
     &= \sum_{\gamma_1,\dots,\gamma_\ell\in\ft{G}^m}
     t_M(f_1^{(\gamma_1)},\dots,f_k^{(\gamma_1)})\cdots t_M(f_1^{(\gamma_\ell)},\dots,f_k^{(\gamma_\ell)})\\
     &= \left(\sum_{\gamma\in\ft{G}^m}t_M(f_1^{(\gamma)},\dots,f_k^{(\gamma)})\right)^\ell\\
     &=\left(\sum_{\gamma\in\ft{G}^m}\sum_{\xi\in \ft{G}^m}\mathbf{1}\big[M^t\xi=L^t\gamma)\big]\right)^\ell.
\end{align*}

Since both $M$ and $L$ have linearly independent rows, $M^t$ and $L^t$ are injective, so the double sum above is $\vert \im(L^t)\cap \im(M^t)\vert^n$, where $\im(L^t)$ and $\im (M^t)$ are the images of $\mathbb{F}_q^m$ under $L^t$ and $M^t$, respectively. Likewise, we have

\begin{align*}
\sum_{\Gamma\in(\ft{G}^m)^\ell}t_M(g_{\pi(1)}^{(\Gamma)},\dots,g_{\pi(k)}^{(\Gamma)})&=\vert \im(\pi(L)^t)\cap \im (M^t)\vert^{n\ell},\\
\sum_{\Gamma\in(\ft{G}^m)^\ell}t_L(g_{\pi(1)}^{(\Gamma)},\dots,g_{\pi(k)}^{(\Gamma)})&=\vert \im(\pi(L)^t)\cap \im (L^t)\vert^{n\ell}.
\end{align*}

If $L$ and $M$ are non-isomorphic, then, for all $\pi$, the rows of $\pi(L)$ and $M$ span different spaces, so that $\im(\pi(L)^t)\cap \im(M^t)$ has dimension at most $m-1$. Thus,
\begin{align*}
\sum_{\Gamma\in(\ft{G}^m)^\ell}\tau_M(g_1^{(\Gamma)},\dots,g_k^{(\Gamma)})=
     \sum_{\pi\in S_k}\sum_{\Gamma\in(\ft{G}^m)^\ell}t_M(g_{\pi(1)}^{(\Gamma)},\dots,g_{\pi(k)}^{(\Gamma)})
     \leq k!q^{n(m-1)\ell}.
\end{align*}
On the other hand, $t_L(g_1^{(\Gamma)},\dots,g_k^{(\Gamma)})=|\im L^t|^{n\ell} = q^{nm\ell}$, so, as $t_L(g_{\pi(1)}^{(\Gamma)},\dots,g_{\pi(k)}^{(\Gamma)})\geq 0$ for all $\pi$, we have
\begin{align*}
    \sum_{\Gamma\in(\ft{G}^m)^\ell}\tau_L(g_1^{(\Gamma)},\dots,g_k^{(\Gamma)})\geq
     \sum_{\Gamma\in(\ft{G}^m)^\ell}t_L(g_1^{(\Gamma)},\dots,g_k^{(\Gamma)})
     \geq q^{nm\ell}.
\end{align*}
Therefore, choosing $\ell$ such that $q^{n\ell}>k!$ gives 
\begin{align*}
\sum_{\Gamma\in(\ft{G}^m)^\ell}\tau_M(g_1^{(\Gamma)},\dots,g_k^{(\Gamma)})<\sum_{\Gamma\in(\ft{G}^m)^\ell}\tau_L(g_1^{(\Gamma)},\dots,g_k^{(\Gamma)}),
\end{align*}
so there must exist $\Gamma\in (\ft{G}^m)^\ell$ such that $g_1^{(\Gamma)},\ldots,g_k^{(\Gamma)}$ are as in the claim.
\end{proof}

\Cref{thm:isom} now follows immediately. 
Indeed, if $t_L(f)=t_M(f)$ for every non-negative $f$, then \Cref{lem:boost} implies that $\tau_L(f_1,\dots,f_k)=\tau_M(f_1,\dots,f_k)$ for any complex-valued functions $f_1,\dots,f_k$, contradicting the conclusion of \Cref{lem:distinguish} if $L$ and $M$ are not isomorphic.

\Cref{thm:isom} has some interesting applications. 
First, we use it to prove that every weakly norming system $L$ is \emph{variable-transitive} in the sense that deleting any variable always leaves an isomorphic system. This is the arithmetic analogue, and closely follows the proof, of a result of Sidorenko~\cite[Lemma~6]{Sid20} saying that if $H$ is weakly norming, then all of the graphs formed by deleting a single edge from $H$ are isomorphic. We note that this is slightly weaker than Sidorenko's full edge-transitivity result, which requires an additional argument that we do not currently see how to transfer to the arithmetic setting. We do however expect that such a result should hold, that is, that for any weakly norming system and any two variables $x_i$ and $x_j$ there is an automorphism of the system that sends $x_i$ to $x_j$. 

\begin{lemma} \label{lem:weaktrans}
    Let $L$ be an $m\times k$ weakly norming system. For a non-negative function  $f:G\rightarrow \mathbb{R}_{\geq 0}$, let $\alpha_i=t_L(f_1,\ldots,f_k)$,
    where $f_i=1$ and $f_j=f$ for all $j\neq i$. Then $\alpha_i=\alpha_j$ for all $i,j\in [k]$.
\end{lemma}

\begin{proof}
    Since $t_L(\cdot)$ and all $\alpha_i$ depend continuously on $f$ and strictly positive functions are a dense subset of the non-negative functions, we may assume that $f>0$. Furthermore, by rearranging edge indices, it is enough to prove that $\alpha_1=\alpha_2$. For $\varepsilon\in \RR$ with $|\varepsilon|\leq \min_{x\in G} f(x)$,
    \begin{align*}
        t_L(f-\varepsilon, f+\varepsilon,f,f,\dots,f) = t_L(f) +(\alpha_2-\alpha_1)\varepsilon - t_L(1,1,f,\dots,f)\varepsilon^2
    \end{align*}
    and, therefore,
    \begin{align}\label{eq:pm-epsilon}
        t_L(f-\varepsilon, f+\varepsilon,f,f,\dots,f)^k=t_L(f)^k+k(\alpha_2-\alpha_1)t_L(f)^{k-1}\varepsilon+O(\varepsilon^2).
    \end{align}
    At the same time, we know from \Cref{weakly_Holder} that
    \begin{align*}
        t_L(f-\varepsilon, f+\varepsilon,f,f,\dots,f)\leq \|f\|_{L}^{k-2}\|f-\varepsilon\|_L\|f+\varepsilon\|_L.
    \end{align*}
    Moreover, $t_L(f\pm \varepsilon)=t_L(f)\pm \alpha \varepsilon+O(\varepsilon^2)$ for $\alpha:=\alpha_1+\dots+\alpha_k$, 
    so that 
    \begin{align*}
        t_L(f-\varepsilon, f+\varepsilon,f,f,\dots,f)^k &\le t_L(f)^{k-2}t_L(f-\varepsilon)t_L(f+\varepsilon)\\ &= t_{L}(f)^{k-2}(t_{L}(f)-\alpha\varepsilon +O(\varepsilon^2))(t_{L}(f)+\alpha\varepsilon +O(\varepsilon^2))\\
        &=t_L(f)^k +O(\varepsilon^2).
    \end{align*}
    If $\alpha_2\neq \alpha_1$, then \eqref{eq:pm-epsilon} contradicts this inequality for $\varepsilon$ of the same sign as $\alpha_2-\alpha_1$ and sufficiently small in absolute value.
\end{proof}

Given an $m \times k$ linear system $L$ and $i \in [k]$, we let $L\setminus i$ be the induced subsystem obtained by deleting the variable $x_i$. Combining~\Cref{lem:weaktrans} with \Cref{prop:subsystem} implies that $t_{L\setminus i}(f) = t_{L\setminus j}(f)$ for all functions $f:\mathbb{F}_q^n\rightarrow\mathbb{R}_{\geq 0}$, which, by~\Cref{thm:isom}, implies the promised analogue of Sidorenko's edge-transitivity theorem.

\begin{corollary}\label{cor:transitivity}
If $L$ is an $m \times k$ weakly norming linear system, then the linear subsystems $L\setminus i$ are isomorphic for all $i \in [k]$.
\end{corollary}

As an immediate consequence of this corollary, we have the following result.

\begin{corollary}\label{cor:zero-column}
    If a non-zero system $L$ is weakly norming, then no column of $L$ is zero.
\end{corollary}

Another application of~\Cref{thm:isom} is that it allows us to prove an analogue of~\cite[Theorem~1.2]{GHL22}, which says that, aside from isolated vertices, each component of a weakly norming graph is isomorphic. 
Given an $m\times k$ linear system $L$, suppose that there exists a partition $I_1\cup I_2\cup \cdots \cup I_r$ of $[k]$
such that $\Sol(L)=\Sol(L_1)\oplus \cdots\oplus \Sol(L_r)$, where $L_j$ is the induced subsystem on $I_j$.
Then $t_L(f)=t_{L_1}(f)\cdots t_{L_r}(f)$ for every $f:G\rightarrow\mathbb{R}$.
If there is no finer partition that decomposes $\Sol(L)$, then each $L_i$ is said to be a~\emph{component} of $L$.

\begin{theorem}
    Let $L_1$ and $L_2$ be two components of an $m\times k$ weakly norming system $L$. Then $L_1$ and $L_2$ are isomorphic weakly norming systems. Furthermore, each component is norming if $L$ is.
\end{theorem}

We omit the proof, which uses the domination inequality in~\Cref{cor:domination}, but is otherwise exactly the same as the simplified proof of~\cite[Theorem~1.2]{GHL22} given in~\cite[Lemma~2.4]{CL23}.

\section{Weakly norming systems of low rank} \label{sec:rank2}

By Proposition~\ref{prop:fourier}, if $L$ is a $1\times k$ system, then
\begin{align*}\label{eq:single_fourier}
    t_L(f_1,\dots,f_k) 
    =\sum_{\xi\in \ft{G}} \ft{f_1}(L_{11}\xi)\cdots \ft{f_k}(L_{1k}\xi).
\end{align*}
In particular, if each $L_{1i}$ is either $a$ or $-a$ for some non-zero $a\in \mathbb{F}_q$ and $\vert\{i\in [k]:L_{1i}=a\}\vert=\vert\{i\in [k]:L_{1i}=-a\}\vert$, then $t_L(f)=\|\ft{f}\|_{\ell^{k}}^k$. As this $\ell^k$-norm of the spectrum is often called the \emph{Schatten--von Neumann norm}, we say that a vector $(a_1,\ldots,a_k)$ is \emph{Schatten} if there exists $a\in \mathbb{F}_q$ such that each $a_i\in \{0,a,-a\}$ and $\vert\{i\in [k]:a_i=a\}\vert=\vert\{i\in [k]:a_i=-a\}\vert$.
In particular, as we have seen above, if a Schatten vector $(a_1,\ldots,a_k)$ has no zero entries, then $(a_1,\ldots,a_k)$ is norming as a linear system. One of the main results of this section is that all weakly norming linear systems must have at least one Schatten vector in their row space.

To state this result, recall that the \emph{girth} of a linear system $L$ is the minimum size of the support of a non-zero vector in the row space of $L$. Furthermore, we write $\mu(L)$ for the set of non-zero vectors in the row space of $L$ with minimum support and $s(L)$ for the number of normalised Schatten vectors in $\mu(L)$, i.e., Schatten vectors whose first non-zero coefficient is $1$.

\begin{theorem}\label{thm:shortest}
    Let $L$ be an $m\times k$ weakly norming system and let $\ell$ be the girth of $L$. Then either all vectors in $\mu(L)$ are Schatten or $s(L)\geq k/\ell$.
\end{theorem}

\begin{proof}
    Since $L$ is weakly norming, $L$ is Sidorenko by \Cref{Sidorenko} and, therefore, $\ell$ is even by \Cref{prop:even-girth}. Suppose that $\mu(L)$ contains a vector $v$ that is not Schatten. Without loss of generality, we may assume that $v$ is of the form $(a_1,\ldots,a_\ell,0\ldots,0)$ for some non-zero $a_1,\ldots,a_\ell$.
    
    Suppose first that the characteristic of $\mathbb{F}_q$ is not 2. Let $\e>0$, let $\gamma$ be a non-zero element of $\ft{G}$ and let $z$ be a complex number with $\vert z\vert =1$. For each $i\in [\ell]$, take $f_i$ to be the function whose Fourier transform is given by
    \begin{equation*}
                \ft{f_i}(\eta) = \begin{cases}
                    1&\text{if $\eta=0$,}\\
                    \e z &\text{if $\eta=(-1)^i a_i\gamma$,}\\
                    \e \conj{z} &\text{if $\eta=(-1)^{i+1} a_i\gamma$,}\\
                    0&\text{otherwise}.\\
                    \end{cases}
    \end{equation*}
    Observe that since $\hat{f_i}(\eta)=\conj{\hat{f_i}(-\eta)}$ for all $\eta$, $f_i$ is real. Furthermore, if $\e$ is sufficiently small, $f_i$ is positive.
    
    By \Cref{prop:fourier}, we have
    \begin{align*}
        t_L(f_1,\ldots,f_\ell,1,\ldots,1)=\sum_{\xi \in \hat{G}^m}\prod_{j=1}^\ell \ft{f_j}\left(L^t_j\xi\right) \prod_{j=\ell+1}^k \one[L^t_j\xi=0].
    \end{align*}
    Since $\ell$ is the girth of $L$, if $\xi\neq 0$ and $L^t_j\xi=0$ for each $j>\ell$, then we must have $L^t_j\xi\neq 0$ for all $j\leq \ell$. Thus, the only non-zero $\xi$ that make a non-zero contribution to the sum above are those for which $L^t_j\xi\in \{\pm a_j\gamma\}$ for each $j\leq \ell$. Since $v$ is in the row space of $L$, there must be $\xi \in \hat{G}^m$ such that $(L^t\xi)_j=v_j\gamma$ for all $i\in [k]$. Furthermore, the only other vectors in the row space of $L$ whose support is in $[\ell]$ are the multiples of $v$. Indeed, if there were a vector $w\in \mu(L)$ with support $[\ell]$ that is not a multiple of $v$, there would be a non-zero linear combination of $v$ and $w$ whose support is a proper subset of $[\ell]$. Thus, if $\prod_{j=1}^\ell \ft{f_j}\left(L^t_j\xi\right)\neq 0$, then $L^t\xi=\pm (a_1\gamma, \ldots,a_\ell \gamma, 0,\ldots, 0)$, so that
    \begin{equation*}
        t_L(f_1,\ldots,f_\ell,1,\ldots,1)=1+2\e^\ell.
    \end{equation*}
    Taking the $k^{\text{th}}$ power, we obtain
    \begin{align}\label{eq:rainbow_shortest}
        t_L(f_1,\ldots,f_\ell,1,\ldots,1)^k\geq 1+2k\varepsilon^\ell.
    \end{align}
    On the other hand, by \Cref{Holder}, we have $t_L(f_1,\ldots,f_\ell,1,\ldots,1)^k\leq \|f_1\|_L^k\cdots\|f_\ell\|_L^k$, where each $\|f_i\|_L^k$ is equal to  
    \begin{equation}\label{eq:Schatten_Xi_count}
        t_L(f_i,\ldots,f_i)=\sum_{\xi \in \hat{G}^m}\prod_{j=1}^k \ft{f_i}\left(L^t_j\xi\right)=1+\sum_{\xi\in X_i} \prod_{j=1}^k \ft{f_i}\left(L^t_j\xi\right) + O(\e^{\ell+1})
    \end{equation}
    with 
    \begin{equation*}
        X_i=\{\xi\in \ft{G}^m\colon L^t_j\xi\in \{0,\pm a_i\gamma\} \text{ for all } j\in [k] \text{ and } \vert\{j\in [k]: L^t_j\xi \neq 0\}\vert =\ell\}.
    \end{equation*}
    
    For $\psi\in \ft{G}^k$, let
    \begin{equation*}
        \Delta_i(\psi)=\vert\{j\in [k]: \psi_j = (-1)^ia_i\gamma\}\vert-\vert\{j\in [k]: \psi_j = (-1)^{i+1}a_i\gamma\}\vert
    \end{equation*}
    and observe, by the definition of $f_i$, that
    \begin{equation*}
        \sum_{\xi\in X_i}\prod_{j=1}^k \ft{f_i}\left(L^t_j\xi\right)=\sum_{\xi\in X_i}\e^\ell z^{\Delta_i(L^t \xi)}=\sum_{\xi\in X_1}\e^\ell z^{\Delta_1(L^t \xi)},
    \end{equation*}
    where the last inequality holds because $X_{i_1}=a_{i_1}a_{i_2}^{-1} X_{i_2}$ for all $i_1, i_2\in [\ell]$. Let 
    $$S_1=\{\xi \in X_1:\Delta_1(L^t\xi)=0\}$$ 
    and note that each normalised Schatten vector $v\in \mu(L)$ gives rise to exactly two elements of $S_1$. Indeed, if $u\in \mathbb{F}_q^m$ is such that $L^tu=v$, then $\pm (a_iu_1\gamma,\ldots,a_iu_m\gamma)\in S_1$.
    
    On the other hand, since
    \begin{equation*}
        \int_{\vert z\vert=1}\sum_{\xi\in X_1\setminus S_1}  z^{\Delta_1(L^t\xi)} \mathrm{d}z=\sum_{\xi\in X_1\setminus S_1}\int_{\vert z\vert=1} z^{\Delta_1(L^t\xi)} \mathrm{d}z=0,
    \end{equation*}
    we may choose $z$ such that the real part of $\sum_{\xi\in X_1\setminus S_1} z^{\Delta_1(L^t \xi)}$ is zero. 
    Inserting this into \eqref{eq:Schatten_Xi_count} yields
    \begin{equation*} 
        t_L(f_i,\ldots,f_i)= 1+\sum_{\xi\in S_1} \prod_{j=1}^k \ft{f_i}\left(L^t_j\xi\right) + O(\e^{\ell+1})=1+2s(L)\e^\ell+O(\e^{\ell+1}).
    \end{equation*}
    Noting that the right-hand side does not depend on $i$, we can combine this with \eqref{eq:rainbow_shortest} to obtain
    \begin{align*}
        1+2k\varepsilon^\ell\leq t_L(f_1,\ldots,f_\ell,1,\ldots,1)^k \leq \|f_1\|_L^k\cdots\|f_\ell\|_L^k\leq 1+2\ell s(L)\varepsilon^\ell+O(\varepsilon^{\ell+1}).
    \end{align*}
    By taking $\varepsilon$ sufficiently small, we see that $s(L)\geq k/\ell$, as desired.

    If $\mathbb{F}_q$ has characteristic 2, the proof is essentially the same, except that we choose $z=1$, noting that $S_1=X_1$ trivially, as $a_i=-a_i$ for all $i\in [\ell]$. 
\end{proof}

As a corollary of \Cref{thm:shortest}, we obtain a classification of all weakly norming systems consisting of a single equation.

\begin{theorem}\label{thm:norming single equation}
A $1\times k$ system $L$ is weakly norming if and only if it is Schatten with no zero entries.
\end{theorem}
\begin{proof}
    We have already seen that such a system is norming. The opposite direction follows from first applying \Cref{cor:zero-column}, which tells us that all entries are non-zero, and then \Cref{thm:shortest}.
\end{proof}

Another corollary of~\Cref{thm:shortest} is that there are at least two distinct Schatten vectors in $\mu(L)$ unless $L$ consists of a single equation.
\begin{corollary}\label{cor:2Schatten}
    Let $L$ be an $m\times k$ weakly norming system with $m\geq 2$. Then there are at least two independent Schatten vectors in $\mu(L)$.
\end{corollary}
\begin{proof}
    By~\Cref{thm:shortest}, it is enough to show that $\ell < k$, where $\ell$ is the girth of $L$. Assume instead that $k=\ell$. Then every vector in the row space of $L$ has support of size $k=\ell$. However, if $m \ge 2$, by using row operations, one can clearly make a vector with smaller support, contradicting that the system has girth $k$.
\end{proof}

Given an $m\times k$ system $L$, consider the $m\times 2k$ system obtained by replacing each column $v$ of the matrix $L$ with two columns $v$ and $-v$. We call this new linear system the \emph{subdivision} of $L$ and denote it by $\sub(L)$. For instance, subdividing the equation $x_1=x_2$ that defines the $L^2$-norm $(\EE_{x\in G} f(x)^2)^{1/2}$ gives the equation $x_1-x_2+x_3-x_4=0$ that defines the Gowers $U^2$-norm.
More generally, as we now show, subdivision preserves the (weakly) norming property (though it is perhaps worth stressing that this is not the case for graphs). 

\begin{proposition} \label{prop:1sub}
    Let $L$ be a weakly norming (resp.~norming) system. Then the subdivision $\sub(L)$ of $L$ is also weakly norming (resp.~norming).
\end{proposition}

\begin{proof}
    Let $L$ be an $m\times k$ norming system. For brevity, write $M=\sub(L)$.
    Our goal is to show that
    \begin{align}\label{eq:fourier_holder_sub}
    \left\vert \sum_{\xi \in \ft{G}^m} \ft{f_1}\left(M_1^t\xi\right)\cdots\ft{f_{2k}}\left(M_{2k}^t\xi\right) \right\vert^{2k} \leq 
    \prod_{j=1}^{2k} \left\vert \sum_{\xi \in \ft{G}^m} \ft{f_j}\left(M_1^t\xi\right)\cdots \ft{f_j}\left(M_{2k}^t\xi\right)    \right\vert
\end{align}
for the $m\times 2k$ system $M$ and any real-valued functions $f_1, \dots, f_{2k}$. For each $j\in [k]$, let $g_j$ be the function on $G$ that satisfies $\ft{g_j}(\xi)=\ft{f_{2j-1}}(\xi)\overline{\ft{f_{2j}}(\xi)} $. Then each $g_j$ is real-valued, since $\ft{g_j}(-\xi) = \ft{f_{2j-1}}(-\xi)\overline{\ft{f_{2j}}(-\xi)} = \overline{\ft{g_j}(\xi)}$. Furthermore, for all $i,j\in [k]$,
\begin{align*}
\ft{g_j}(L_i^t\xi) = \ft{f_{2j-1}}(L_i^t\xi)\overline{\ft{f_{2j}}(L_i^t\xi)}   = \ft{f_{2j-1}}(M_{2i-1}^t\xi)\ft{f_{2j}}(M_{2i}^t\xi),
\end{align*}
where the second equality follows from the definition of $\sub(L)$. 

Rewriting the left-hand side of~\eqref{eq:fourier_holder_sub} in terms of the $g_j$ and applying~\eqref{eq:fourier_holder} gives
\begin{align}\label{eq:g_holder}
     \nonumber\left\vert\sum_{\xi \in \ft{G}^m} \ft{f_1}\left(M_1^t\xi\right)\cdots\ft{f_{2k}}\left(M_{2k}^t\xi\right)\right\vert^{2k}  &=\left\vert \sum_{\xi \in \ft{G}^m} \ft{g_1}(L_1^t\xi)\cdots\ft{g_k}(L_k^t\xi)\right\vert^{2k}\\
     & \leq 
     \prod_{j=1}^{k} \left\vert \sum_{\xi \in \ft{G}^m} \ft{g_j}\left(L_1^t\xi\right)\cdots \ft{g_j}\left(L_{k}^t\xi\right)    \right\vert^2.
\end{align}
By the Cauchy--Schwarz inequality,
\begin{align*}
    &\left\vert\sum_{\xi \in \ft{G}^m} \ft{g_j}\left(L_1^t\xi\right)\cdots \ft{g_j}\left(L_{k}^t\xi\right) \right\vert^2 =
    \left\vert\sum_{\xi \in \ft{G}^m} \prod_{i=1}^{k}\ft{f_{2j-1}}\left(L_{i}^t\xi\right)
    \ft{f_{2j}}\left(-L_{i}^t\xi\right)\right\vert^2
    \\&=
    \left\vert\sum_{\xi \in \ft{G}^m} \prod_{i=1}^{k}\ft{f_{2j-1}}\left(M_{2i-1}^t\xi\right)
    \ft{f_{2j}}\left(M_{2i}^t\xi\right)\right\vert^2
    \\
    &\leq \left(\sum_{\xi \in \ft{G}^m} \prod_{i=1}^{k}\left\vert\ft{f_{2j-1}}\left(M_{2i-1}^t\xi\right)\right\vert^2\right)
    \left(\sum_{\xi \in \ft{G}^m} \prod_{i=1}^{k}\left\vert\ft{f_{2j}}\left(M_{2i}^t\xi\right)\right\vert^2\right)\\
    &= \left(\sum_{\xi \in \ft{G}^m} \prod_{i=1}^{k}\ft{f_{2j-1}}\left(M_{2i-1}^t\xi\right)\ft{f_{2j-1}}\left(M_{2i}^t\xi\right)\right)
     \left(\sum_{\xi \in \ft{G}^m} \prod_{i=1}^{k}\ft{f_{2j}}\left(M_{2i-1}^t\xi\right)\ft{f_{2j}}\left(M_{2i}^t\xi\right)\right)\\
     &= \left(\sum_{\xi \in \ft{G}^m} \prod_{i=1}^{2k}\ft{f_{2j-1}}\left(M_{i}^t\xi\right)\right)
     \left(\sum_{\xi \in \ft{G}^m} \prod_{i=1}^{2k}\ft{f_{2j}}\left(M_{i}^t\xi\right)\right).
\end{align*}
Substituting this into~\eqref{eq:g_holder} proves~\eqref{eq:fourier_holder_sub}.
For weakly norming systems $L$, the only difference is that we need to verify that $g_j$ is non-negative whenever both $f_{2j-1}$ and $f_{2j}$ are. But $g_j=f_{2j-1}*f_{2j}^{-}$, where $f_{2j}^{-}(x)=f_{2j}(-x)$, so $g_j$ is always non-negative.
\end{proof}

This proposition easily generalises to the $r$-subdivision $\sub_r(L)$ of $L$, where each column $v$ of $L$ is replaced by $r$ copies of $v$ and $-v$, respectively. 
To see this, we only need to replace the use of the Cauchy--Schwarz inequality in the proof of \Cref{prop:1sub} above by H\"older's inequality.

\begin{proposition}\label{prop:subdivide}
    Let $L$ be a weakly norming (resp.~norming) system. Then the $r$-subdivision $\sub_r(L)$ of $L$ is also weakly norming (resp.~norming).
\end{proposition}
 
In the light of this proposition, one may easily construct $m\times 2rm$ norming systems for any positive integers $r$ and $m$.

\begin{example}\label{ex:L^p_by_identity}
    Let $L$ be the $m\times m$ system given by the identity matrix. This is a degenerate case where $t_L(f)=\EE[f]^m$, but subdividing $L$ gives more interesting examples. For instance, $\sub(L)$ is the $m\times 2m$ system given by the $m\times 2m$ matrix
    \begin{align*}
        \begin{pmatrix}
            1 & -1 &  &  & \cdots &  &  & \\
             &  & 1 & -1 & \cdots &  &  & \\
            &&&&\ddots\\
             &  &  &\cdots & 1 & -1 & &  \\
             &  &  &\cdots &  &  & 1 & -1 \\
        \end{pmatrix}.
    \end{align*}
    As $L$ is semi-norming, it immediately follows that $\sub(L)$ is as well. 
    Concretely, since $\Sol(L)$ consists of the set of solutions to $x_1=x_2$, $x_3=x_4$, $\dots$, $x_{2m-1}=x_{2m}$, $t_{\sub(L)}(f) = \EE [f(x)^2]^m$. That is, the $\sub(L)$-norm is just the $L^2$-norm. More generally, the $r$-subdivision of $L$ 
    yields the $L^{2r}$-norm.
\end{example}

We can also construct $(m-1)\times 2rm$ (weakly) norming systems for any positive integers $r$ and~$m$.

\begin{example}\label{ex:theta}
    Let $L$ be the $(m-1)\times m$ system given by the matrix
    \begin{align*}
        \begin{pmatrix}
            1 & -1 &  &   \cdots &  & \\
             & 1 & -1 &   \cdots &  & \\
            &&&\ddots\\
             &  & &\cdots & 1 & -1 &  \\
             &  &  &\cdots &  & 1 & -1 \\
        \end{pmatrix}.
    \end{align*}
    Then $\Sol(L)$ consists of the set of solutions to $x_1=x_2=\dots=x_{m}$, so that $t_L(f) = \EE[f(x)^{m}]$. Thus, if $m$ is even, then $L$ is norming and otherwise it is weakly norming.
    To see why it is not norming for odd $m$, it is enough to construct nonzero $f$ such that $\EE[f(x)^m]=0$. Indeed, such a function exists, e.g., $f(x)=\mathbf{1}[x=0]-\mathbf{1}[x=a]$ for a nonzero $a$. 
    \Cref{prop:subdivide} implies that the $r$-subdivision of $L$ is also weakly norming. 
    In particular, the $2\times 6$ weakly norming system in \Cref{ex:K23} corresponds to the case where $m=3$ and $r=1$.
\end{example}

It turns out that the two examples above include all possible $2\times k$ weakly norming systems.

\begin{theorem}
    A $2\times k$ system $L$ is weakly norming if and only if it is isomorphic to one of the following three systems:
    \begin{enumerate}[(i)]
        \item $k=3$ and $x_1=x_2=x_3$.
        \item $k=4r$ for some $r\in \mathbb{N}$, $x_1+\dots+x_r=x_{r+1}+\dots+x_{2r}$ and $x_{2r+1}+\dots+x_{3r}=x_{3r+1}+\dots+x_{4r}$.\label{it:disjoint}
        \item $k=6r$ for some $r\in \mathbb{N}$ and $x_1-x_2+\cdots-x_{2r}=x_{2r+1}-x_{2r+2}+\cdots-x_{4r}=x_{4r+1}-x_{4r+2}+\cdots-x_{6r}$.\label{it:intersect}
    \end{enumerate}
\end{theorem}

\begin{proof}
    The first one is simply the classical $L^3$-norm.
    The second is the $r$-subdivision of the $2\times 2$ identity matrix $L$.
    The third one is isomorphic to the $r$-subdivision of the first one.
    \Cref{prop:subdivide} therefore proves that all three systems are weakly norming.

    Let $L$ be a $2\times k$ weakly norming system with girth $\ell$. By \Cref{cor:2Schatten}, the row space of $L$ contains two distinct Schatten vectors $a$ and $b$ whose support has size $\ell$ and whose first non-zero coordinate is $1$. These two vectors are linearly independent and, therefore, span the whole row space. Let $I$ be the intersection of the support of $a$ and $b$. If $I$ is empty, then solutions to $L$ are, up to permutation of coordinates, of the form in \ref{it:disjoint}. 

    Suppose now that $I$ is non-empty. Without loss of generality, we may assume that the support of $a$ is $[1,\ell]$ and the support of $b$ is $[k-\ell+1,k]$. Note that $k+|I|=2\ell$. Recall that $L\setminus j$ is defined as the linear system in which we remove all vectors with non-zero $j^\text{th}$ coefficient from the row space of $L$. Since $L\setminus k$ contains the equation $a_1x_1+\dots+a_\ell x_\ell=0$, we know by \Cref{cor:transitivity} that $L\setminus \ell$ also contains a Schatten vector whose support has size $\ell$. Thus, the row space of $L$ must contain a Schatten vector $c$ with $c_\ell=0$. We can express $c$ as $c=\lambda a+ \mu b$ for $\lambda,\mu \in \mathbb{F}_q$. Since we may assume without loss of generality that $c_1=1$ and, therefore, $c_i\in \{0,\pm 1\}$ for all $i\in [k]$, it follows that $\lambda=1$ and $\mu\in \{\pm 1\}$. Furthermore, by replacing $b$ by $-b$ if necessary, we may assume that $\mu=1$. It follows that $c_i=0$ precisely when $i\in I$. Thus, $\ell = k-|I|$, so, since $k+|I|=2\ell$, we have $2k=3\ell=6|I|$. 
    
    Denoting $|I|=k/3$ by $s$, we see that because 
    $a_1+\dots+a_s+b_{2s+1}+\dots+b_{3s}=c_1+\dots+c_k=0$, we have $a_1+\dots+a_s=-(b_{2s+1}+\dots+b_{3s})$. We also have $a_i=-b_i$ for all $i\in [s+1,2s]$ and, since all rows sum to zero, we know that $a_{s+1}+\dots+a_{2s}=-(a_1+\dots+a_s)$. In summary, as all non-zero coefficients of $a$ and $b$ are $1$ or $-1$, we may assume by reordering columns if necessary that $L$ has the form
    \begin{align*}
        L=\begin{pmatrix}
        a_1&\cdots&a_s&-a_1&\cdots&-a_s&0&\cdots&0\\
        0&\cdots&0&a_1&\cdots&a_s&-a_1&\cdots&-a_s
        \end{pmatrix}
    \end{align*}
    for some $a_1,\ldots,a_s\in \{\pm 1\}$. To see that $L$ is a system of the form in \ref{it:intersect}, it remains only to show that $\sigma=\sum_{i=1}^s a_i=0$. 
    
    Assume, for the sake of contradiction, that $\sigma\neq 0$. We fix linearly independent $\gamma_1,\gamma_2\in \mathbb{F}_q^n$ and let $\gamma_3=\gamma_2-\gamma_1$. 
    We also take some $\alpha>0$ and set $z=\exp(\pi i/\sigma)$. 
    We define $f$ and $g$ to be the real-valued functions on $\mathbb{F}_q^n$ with Fourier coefficients
    \begin{align*}
            \ft{f}(\xi)&=\begin{cases}
            1 &\text{if } \xi = 0,\\
            \alpha z^{\varepsilon} &\text{if $\xi = \varepsilon\gamma_1$ for $\varepsilon\in \{\pm 1\}$},\\
            \alpha &\text{if $\xi \in \{\pm \gamma_2,\pm \gamma_3\}$},\\
            0 &\text{otherwise, }
        \end{cases}\\
        \ft{g}(\xi)&=\begin{cases}
            1 &\text{if } \xi = 0,\\
            \alpha z^{(1-a_s\sigma)\varepsilon} &\text{if $\xi = \varepsilon\gamma_1$ for $\varepsilon\in \{\pm 1\}$},\\
            \alpha &\text{if $\xi \in \{\pm \gamma_2,\pm \gamma_3\}$},\\
            0 &\text{otherwise.}
        \end{cases}
    \end{align*}
    Note that $f$ and $g$ take only positive real values if $\alpha$ is sufficiently small. Moreover, for any $\xi \in \hat{G}$, we have that 
    \begin{align*}
        \ft{g}(a_s\xi)\prod_{i=1}^{s-1}\ft{f}(a_i\xi)=\left|\ft{f}(\xi)\right|^s,
    \end{align*}
    as is readily verified by checking the only non-trival case $\xi=\pm \gamma_1$. Letting $f_s=f_{2s}=f_{3s}=g$ and $f_i=f$ for all other $i\in [k]$, we therefore have, by \Cref{prop:fourier}, that
    \begin{align*}
        t_L(f_1,\ldots,f_k)&=\sum_{\xi_1,\xi_2} \left(\ft{g}(a_s\xi_1)\prod_{i=1}^{s-1}\ft{f}(a_i\xi_1)\right)
        \left(\ft{g}(a_s(\xi_2-\xi_1))\prod_{i=1}^{s-1}\ft{f}(a_i(\xi_2-\xi_1)\right)
        \overline{\left(\ft{g}(a_s\xi_2)\prod_{i=1}^{s-1}\ft{f}(a_i\xi_2)\right)}\\
        &=\sum_{\xi_1,\xi_2}|\ft{f}(\xi_1)|^s|\ft{f}(\xi_2-\xi_1)|^s|\ft{f}(\xi_2)|^s.
    \end{align*}
    On the other hand, denoting the real part of a complex number $z$ by $\Re(z)$, we have that 
    {\allowdisplaybreaks
    \begin{align*}
        t_L(f)&=\sum_{\xi_1,\xi_2}\prod_{j=1}^{s}\ft{f}(a_j\xi_1)\prod_{j=1}^{s}\ft{f}(a_j(\xi_2-\xi_1))\prod_{j=1}^{s}\ft{f}(-a_j\xi_2)\\
        &= \Re\left(\sum_{\xi_1,\xi_2}\prod_{j=1}^{s}\ft{f}(a_j\xi_1)\prod_{j=1}^{s}\ft{f}(a_j(\xi_2-\xi_1))\prod_{j=1}^{s}\ft{f}(-a_j\xi_2)\right)\\
        &= \sum_{\xi_1,\xi_2} \Re\left(\prod_{j=1}^{s}\ft{f}(a_j\xi_1)\prod_{j=1}^{s}\ft{f}(a_j(\xi_2-\xi_1))\prod_{j=1}^{s}\ft{f}(-a_j\xi_2)\right)\\
        & \leq \sum_{\xi_1,\xi_2} \left|\prod_{j=1}^{s}\ft{f}(a_j\xi_1)\prod_{j=1}^{s}\ft{f}(a_j(\xi_2-\xi_1))\prod_{j=1}^{s}\ft{f}(-a_j\xi_2)\right|\\
        &=\sum_{\xi_1,\xi_2}|\ft{f}(\xi_1)|^s|\ft{f}(\xi_2-\xi_1)|^s|\ft{f}(\xi_2)|^s = t_L(f_1,\ldots,f_k).
    \end{align*}}
    Moreover, the inequality is strict when $f_i=f$, since
    \begin{align*}
        t_L(f)&=\sum_{\xi_1,\xi_2}\prod_{j=1}^{s}\ft{f}(a_j\xi_1)\prod_{j=1}^{s}\ft{f}(a_j(\xi_2-\xi_1))\prod_{j=1}^{s}\ft{f}(-a_j\xi_2)\\
        &\leq t_L(f_1,\ldots,f_k)+\Re\left(\prod_{j=1}^{s}\ft{f}(a_j\gamma_1)\ft{f}(a_j\gamma_3)\ft{f}(-a_j\gamma_2)\right)-|\ft{f}(\gamma_1)|^s|\ft{f}(\gamma_3)|^s|\ft{f}(\gamma_2)|^s\\
        &=t_L(f_1,\ldots,f_k)+\alpha^k\left(\Re(z^\sigma)-1\right).
    \end{align*}
    Since $z=\exp(\pi i /\sigma)$, this yields that $t_L(f)\leq t_L(f_1,\ldots,f_k)-2\alpha^k$, implying that $L$ does not have the Hölder property.
\end{proof}

\begin{remark}
    Which of these weakly norming systems are norming? 
    In~\Cref{ex:theta}, we have already seen that the first system is not norming. 
    The second family of systems $L_2$ is norming, since $t_{L_2}(f)^{1/4r}=t_{L'}(f)^{1/2r}$, where $L'$ is the single Schatten equation $x_1+\cdots+x_r=x_{r+1}+\cdots+x_{2r}$. 
    To verify that the third system $L_3$ is not norming, let $f$ be the real function that satisfies $\ft{f}(\xi)=\mathbf{1}[\xi=\pm a]$ for some $a\neq 0$. Then 
    \begin{align*}
        t_{L_3}(f)=\sum_{\xi_1,\xi_2} |\ft{f}(\xi_1)|^{2r}|\ft{f}(\xi_2)|^{2r}|\ft{f}(\xi_1+\xi_2)|^{2r},
    \end{align*}
    which evaluates to $0$ unless $\mathbb{F}_q$ has characteristic $3$.
    If $\mathbb{F}_q$ has characteristic $3$, then $t_{L_3}(f)=2$.
    However, deleting the first $2r$ variables $x_1,\dots,x_{2r}$ gives a Schatten equation $L'$ of length $4r$, for which $t_{L'}(f)=2$. In either case, we have a  contradiction to~\Cref{cor:domination}.
\end{remark}

One may wonder if weakly norming systems of rank three or more have a similar characterisation. Unfortunately (or fortunately), there are weakly norming systems of rank three that do not come from subdividing trivial examples.

\begin{example} \label{ex:K41sub}
    Let $L$ be the $3\times 12$ system given by the matrix
    \begin{align*}
        \begin{pmatrix*}[r]
            1 & -1 & 1 & -1 &  1 & -1 & 0  & 0  & 0 & 0 & 0 &  0  \\
            0 &  0 & 0 & 0 &  -1 & 1  &-1 & 1 & -1 & 1 & 0  & 0     \\
            -1&  1 & 0  & 0  &  0   &   0 & 1 & -1& 0   & 0  &1  & -1 
        \end{pmatrix*}.
    \end{align*}
    Then $\Sol(L)$ consists of all those $(x_1,x_2,\ldots,x_{12})$ that can be parameterised by
    \begin{align*}
        (x_1,x_2,x_3,x_4,x_5,x_6)&=(y_1 +z_{12}, z_{12} +y_2 , y_2 + z_{23}, z_{23}+y_3, y_3 +z_{13}, z_{13}+y_1),\\
        (x_5,x_6,x_7,x_8,x_9,x_{10}) &= (y_3 +z_{31}, z_{31}+y_1, y_1+z_{14}, z_{14}+y_4,y_4+z_{34},z_{34}+y_3),\\
        (x_1,x_2,x_7,x_8,x_{11},x_{12}) &= (y_1 +z_{12}, z_{12} +y_2, y_1+z_{14}, z_{14}+y_4, y_4+z_{24}, z_{24}+y_2).
    \end{align*}
    This system is weakly norming, since the parameterisation comes from the singleton-pair incidence graph of the 
    set $\{1,2,3,4\}$, which is a weakly norming graph. More details will be given in~\Cref{sec:Cayley}. 

    On the other hand, $L$ is included in neither \Cref{ex:L^p_by_identity} nor \Cref{ex:theta}. Indeed, if \Cref{ex:L^p_by_identity} includes $L$, then it must be the case that $m=3$ and $r=2$; however, the $2$-subdivision of the $3\times 3$ degenerate system has girth 4, whereas the girth of $L$ is 6. Similarly, \Cref{ex:theta} does not include $L$ for the simple reason that one would need to take $m = 4$ and $r = 3/2$. 
\end{example}

We leave the characterisation of weakly norming systems of rank three as an open problem.

\section{Forcing systems} \label{sec:forcing}

A graph $H$ is said to be \emph{forcing} if the fact that $t_H(W)=p^{e(H)}$ for a graphon $W$ with $p=t_{K_2}(W)$ implies that $W=p$ a.e. While it is easy to see that trees and non-bipartite graphs are not forcing, the forcing conjecture~\cite{SkokanThoma04}, a strengthening of Sidorenko's conjecture, says that all bipartite graphs containing a cycle should be forcing. 

In analogy with the graph case, we say that an $m\times k$ linear system $L$ is \emph{forcing} if, for every $n$ and every non-negative function $f$ on $\mathbb{F}_q^n$, $t_L (f) = \EE[f]^{k}$ implies that $f$ is a constant function on $\mathbb{F}_q^n$. Taking our lead from a similar result for graphs, we will now show that every forcing system is Sidorenko.

\begin{proposition}\label{prop:forcing-sidorenko}
    Let $L$ be an $m\times k$ linear system such that $t_L(f)<\EE[f]^k$ for some $f:\mathbb{F}_q^n\rightarrow [0,1]$. Then $L$ is not forcing.
\end{proposition}

\begin{proof}
    We may assume that $n>1$ by extending $f$ to a larger $\mathbb{F}_q^n$ through adding an extra zero to each point in the support. Let $g_1$ be the indicator function of the subspace $\{(x_1,\dots,x_n)\in\mathbb{F}_q^n: x_n=0\}$.
    Then $t_L(g_1) = (q^{n-1})^{k-m}/(q^n)^{k-m} = q^{m-k}$ since $L$ consists of $m$ linearly independent forms and $\EE[g_1]^k=q^{-k}$. Thus, 
    \begin{align*}
        t_L(g_1)-\EE[g_1]^k = q^{-k}(q^m-1)>0.
    \end{align*}
    Now consider $P(\alpha):=t_L(\alpha g_1 +(1-\alpha)f)-\EE[\alpha g_1+ (1-\alpha)f]^k$, which is a polynomial in $\alpha$. As $P(0)<0$ and $P(1)>0$, there exists $\alpha_0\in (0,1)$ with $P(\alpha_0)=0$. Then the function $h_1:=\alpha_0 g_1+(1-\alpha_0)f$ with range in $[0,1]$ satisfies $t_L(h_1)=\EE[h_1]^k$, so, if $L$ is forcing, $h_1$ must be constant.
    
    One can similarly obtain a  function $h_2=\beta_0 g_2+(1-\beta_0)f$ such that $\beta_0 \in (0,1)$ and $t_L(h_2)=\EE[h_2]^k$, where $g_2$ is the indicator function of the subspace  $\{(x_1,\dots,x_n)\in\mathbb{F}_q^n: x_{n-1}=0\}$.
    It follows that $h_2$ is also a constant and, thus, that $g_1$ is a constant multiple of $g_2$ plus another constant; however, it is easy to check that this is not the case.
\end{proof}

We now classify all the single equations that are forcing by showing that a single equation is forcing if and only if it is norming. Recall that, by~\Cref{thm:norming single equation}, the norming and weakly norming properties are the same for a single equation, so all three properties agree in this case. 

\begin{theorem}\label{thm:single_forcing}
    A $1\times k$ system $L$ is forcing if and only if it is norming.
\end{theorem}

\begin{proof}
By~\Cref{thm:norming single equation}, $L$ is norming if and only if it is isomorphic to $\begin{pmatrix}1&\cdots&1&-1&\cdots&-1\end{pmatrix}$. Since $t_L(f) = \sum_\xi |\ft{f}(\xi)|^{2k}$, $t_L(f)=\EE[f]^{2k}$ if and only if $\sum_{\xi\neq 0} |\ft{f}(\xi)|^{2k}=0$.
That is, $\ft{f}(\xi)=0$ for every $\xi\neq 0$ and, thus, $f$ is a constant function.

To prove the converse, suppose $L$ is forcing but not norming. By~\Cref{prop:forcing-sidorenko}, $L$ is Sidorenko and so, by~\cite[Theorem 1.4]{FPZ19}, isomorphic to a system of the form $\begin{pmatrix}a_1&-a_1&\cdots&a_{r}&-a_{r}\end{pmatrix}$ for $r=k/2$ and $a_1,\dots,a_r\in \mathbb{F}_q$. Thus, 
\[
    t_L(f) = \sum_{\xi} |\ft{f}(a_1\xi)|^2\cdots|\ft{f}(a_r\xi)|^2.
\]
Since $L$ is not norming, there must be $i,j\in [r]$ such that $a_i\notin \{\pm a_j\}$. Fix a non-zero $\gamma\in \ft{\mathbb{F}_q^n}$ and let $f$ be the function with Fourier transform $\ft{f}(\xi)=\frac{1}{2}\left(\mathbf{1}[\xi=\gamma]+\mathbf{1}[\xi=-\gamma]\right)$ for every $\xi\neq 0$ and $\ft{f}(0)=1$. Then $f$ is a non-negative real-valued function, as $\sum_{\xi\neq 0}|\ft{f}(\xi)|\leq 1$.  
Moreover, $\ft{f}(a_i\xi)\ft{f}(a_j\xi)=0$ whenever $\xi\neq 0$,
so $t_L(f)=|\ft{f}(0)|^{2k}=\EE[f]^{2k}$. But $f$ is clearly not constant, contradicting our assumption that $L$ is forcing. 
\end{proof}

In particular, every $1\times k$ weakly norming system is forcing. We now prove a much stronger result, that every weakly norming system $L$ is forcing. For weakly norming graphs, we know that they are forcing whenever they contain an even cycle because there is a `domination' relation between the weakly norming graph and the even cycle (see, for example,~\cite[Section 5.2]{CL16}). 
We may then use the simple fact that every even cycle is forcing to reach the required conclusion. 
Unfortunately, this proof does not transfer in an obvious way to the arithmetic setting, so we take a slightly different approach. 

\begin{theorem} \label{thm:weaktoforce}
    Every weakly norming system $L$ is forcing.
\end{theorem}

\begin{proof}
    Let $\ell$ be the girth of $L$. By \Cref{thm:shortest}, the row space of $L$ contains a Schatten vector $a$ whose support has size $\ell$. Without loss of generality, we may assume that the support of this vector is $[\ell]$. Let $M$ be the $1\times \ell$ system $\begin{pmatrix}a_1&\cdots&a_\ell\end{pmatrix}$. 
    
    Suppose now that $t_L(g)=\EE[g]^k$ for some non-negative function $g$ on $\mathbb{F}_q^n$. Because $\ell$ is the girth of $L$, every vector in the row space of $L$ that is only supported on $[\ell]$ must be a multiple of $a$. Thus, by~\Cref{prop:subsystem}, letting $f_i=g$ for $i\leq \ell$ and $f_i=1$ for $i>\ell$, we have $t_M(g) =  t_L(f_1,\dots,f_k)$. By \Cref{weakly_Holder},
    \begin{align*}
       t_L(f_1,\dots,f_k) \leq \prod_{i=1}^k \|f_i\|_L.
    \end{align*}
    But $\|g\|_L=\EE[g]$ by assumption and $\|1\|_L=1$, which yields $t_M(g)\leq \EE[g]^\ell$. Since $M$ is norming, \Cref{thm:single_forcing} implies that it is also forcing, so that $g$ must be constant.
\end{proof}

One might suspect that~\Cref{thm:single_forcing} generalises and the converse of~\Cref{thm:weaktoforce} is also true. However, as we now show, there are forcing systems that are not weakly norming. 

\begin{example}\label{ex:1-subK4}
    Consider the $2\times 7$ linear system $L$ given by the matrix
\begin{align*}
    \begin{pmatrix*}[r]
        1 & 1 & -1 & -1 & 0 & 0 & 0 \\
        0 & 0 & 0 & 1 & 1 & -1 & -1 
    \end{pmatrix*}.
\end{align*}
Let $f$ be a non-negative function. Then the Cauchy--Schwarz inequality gives
\begin{align*}
    t_L(f) = \EE_{x\in G}\left[ f(x) f*f*f(x)^2\right] \geq
    \frac{\EE_{x\in G}\left[f(x)f*f*f(x)\right]^2}{\EE_{x\in G}[f]} = \frac{\EE_{x\in G}\left[f*f(x)^2\right]^2}{\EE_{x\in G}[f]}.
\end{align*}
Now suppose that $t_L(f) = \EE[f]^7$. 
It follows that $\sum_{\xi}\ft{f}(\xi)^4 = \EE\left[f*f(x)^2\right]=\EE[f]^4$, which means $\ft{f}(\xi)=0$ for every $\xi\neq 0$. Therefore, $L$ is forcing.

To see that $L$ is not weakly norming, we apply~\Cref{cor:transitivity}. Indeed, $L\setminus 1$ is isomorphic to the Schatten system $x_1+x_2=x_3+x_4$, whereas $L\setminus 4$ is isomorphic to $x_1+x_2+x_3=x_4+x_5+x_6$.
\end{example}

\section{From norming hypergraphs to norming systems}\label{sec:Cayley}

In this section, we describe how (weakly) norming hypergraphs naturally give rise to (weakly) norming linear systems. Indeed, given a $k$-uniform hypergraph $H$, we may look at the homomorphism density $t_H(W)$ restricted to functions of the form $W(x_1,\dots,x_k) := f(x_1+\cdots+x_k)$ for a function $f:\mathbb{F}_q^n\rightarrow \mathbb{R}$. Then, in this Cayley setting, we have that
\begin{align*}
    t_H(W) = \mathbb{E}\left[\prod_{i_1\cdots i_k\in E(H)}W(x_{i_1},\dots,x_{i_k})\right]=
     \mathbb{E}\left[\prod_{i_1\cdots i_k\in E(H)}f(x_{i_1}+\cdots+x_{i_k})\right] =: t_{L_H}(f),
\end{align*}
where $L_H$ is the system of linear equations whose solution set is defined by the parametrisation 
\begin{align*}
    y_e=x_{i_1}+\cdots+x_{i_k},~~e=\{i_1,\dots,i_k\}\in E(H).
\end{align*}
We call $L_H$ the \emph{linear $H$-system} and the parametrisation above the \emph{standard parametrisation} of $\Sol(L_H)$.

To see that $L_H$ is (weakly) norming whenever $H$ is, it is enough to observe that if $f_1, \dots, f_{e(H)}: \mathbb{F}_q^n\rightarrow \mathbb{R}$ are functions and $W_e(x_1,\dots,x_k) := f_e(x_1+\cdots+x_k)$ for each $e \in E(H)$, then 
\begin{align*}
    t_H(W_1, \dots, W_{e(H)}) & = \mathbb{E}\left[ \prod_{e=i_1\cdots i_k\in E(H)}W_e(x_{i_1},\dots,x_{i_k})\right]\\ 
    & = \mathbb{E}\left[\prod_{e=i_1\cdots i_k\in E(H)}f_e(x_{i_1}+\cdots+x_{i_k}) \right]= t_{L_H}(f_1, \dots, f_{e(H)}).
\end{align*}
Therefore, the inequality~\eqref{eq:rainbow} holds for $L=L_H$ if and only if the corresponding inequality
\begin{align*}
    |t_H(W_1, \dots, W_{e(H)})|\leq \prod_{e\in E(H)}\|W_e\|_H
\end{align*}
holds, where $\|W\|_H:=|t_H(W)|^{1/e(H)}$. To summarise, we have the following result.

\begin{proposition}
    The linear $H$-system $L_H$ is (weakly) norming whenever $H$ is (weakly) norming. 
\end{proposition}

Though our focus in this paper is on norming systems of equations, we note that a similar result holds for Sidorenko and common systems. That is, for example, that any Sidorenko hypergraph gives rise to a Sidorenko system of equations. This gives a simple mechanism for producing many examples of Sidorenko and common systems, though it also suggests that the classification of such systems will be difficult.

We now give some examples of weakly norming systems that come from weakly norming hypergraphs.

\begin{example}\label{ex:Kab}
Let $H=K_{a,b}$, the complete bipartite graph on $A\cup B$ with $|A|=a$ and $|B|=b$,
which is one of the simplest examples of a weakly norming graph. Then
\begin{align*}
    t_{L_H}(f) = \EE\left[\prod_{(a,b)\in A\times B}f(x_a+y_b)\right],
\end{align*}
where the expectation is taken over uniform random vectors $(x_a)_{a\in A}\in G^A$ and $(y_b)_{b\in B}\in G^B$. The corresponding norms are the so-called \emph{grid norms}, which played a key role in the recent paper~\cite{FHHK23}.
In particular, if $a=b=2$, then $t_{L_H}(f)=\EE\left[f(x_1+y_1)f(x_2+y_1)f(x_1+y_2)f(x_2+y_2)\right]$, which corresponds to the fourth power of the Gowers $U^2$-norm.
Moreover, if $a=2$ and $b=3$, then one can check that $L_H$ is isomorphic to the $2\times 6$ linear system given in~\Cref{ex:K23}.
\end{example}

\begin{example}\label{ex:subdiv_graph}
    For $H$-systems, the subdivision operation defined  in~\Cref{sec:rank2} can be viewed in purely graph-theoretic terms. Indeed, if $H$ is a norming $r$-graph, let $H^+$ and $H^-$ be two vertex-disjoint copies of $H$, both also vertex-disjoint from $H$. For an edge $e\in E(H)$, let $e^+$ and $e^-$ be the edges corresponding to $e$ in $H^+$ and $H^-$, respectively. 
    The \emph{$2$-blow-up extension} of $H$, denoted by $\beta^2(H)$, is then the $(r+1)$-graph on $V(H^+)\cup V(H^-)\cup E(H)$, whose edges are those sets of the form $e^+\cup\{e\}$ and $e^-\cup\{e\}$. 
    A standard application of the Cauchy--Schwarz inequality implies that $\beta^2(H)$ is also norming, a fact which first appeared implicitly in~\cite{CHPS12}. 

    We claim that the subdivision $\sub(L_H)$ of $L_H$ is equal to the $\beta^2(H)$-system.
    To see this, enumerate $E(H)$ and $E(\beta^2(H))$ by the integers in $[k]$ and $[2k]$ in such a way that the $(2i-1)$-th and the $2i$-th edges of $\beta^2(H)$ are obtained by adding a common vertex to two copies of the $i$-th edge of $H$, respectively.
    Then each vector $(z_1,z_2,\ldots,z_{2k})$ in $\Sol(L_{\beta^2(H)})$ can be parameterised by $z_{2j-1}=y_j+x_j$ and $z_{2j}=y_j'+x_j$, where 
    $(y_1,\ldots, y_k)$ and $(y_1',\ldots,y_k')$ are vectors in $\Sol(L_{H})$ and $x_j$ is a free variable.
    Therefore, if $(y_1,\ldots, y_k)$ satisfies $a_1y_1+\cdots + a_ky_k=0$ and $(y'_1,\ldots, y'_k)$ satisfies $a_1y'_1+\cdots + a_ky'_k=0$, 
    then $(z_1,z_2,\ldots,z_{2k})$ satisfies the equation $a_1(z_1-z_2)+a_2(z_3-z_4)+\cdots+a_k(z_{2k-1}-z_{2k})=0$, which is in $\sub(L_H)$. 
    Conversely, if $(z_1,z_2,\ldots,z_{2k})$ satisfies a linear equation $b_1z_1+b_2z_2+\cdots +b_{2k}z_{2k}=0$, then $b_{2i-1}=b_{2i}$ must hold, since otherwise the free variable $x_i$ does not vanish. It therefore follows that the equation is in $\sub(L_H)$.

    To give a concrete example, if $H=C_4$, then $\beta^2(H)$ is the line $3$-graph of a cube, denoted by~$M(3)$ in~\cite{CHPS12}. The $H$-system for $H=M(3)$, the line $3$-graph of a cube, is therefore the subdivision of $L_{C_4}$ and is defined by the single linear equation $x_1+x_2+x_3+x_4=y_1+y_2+y_3+y_4$, which is just the $L^8$-norm of the spectrum.
\end{example}

For graphs $H$, it is reasonably straightforward to determine the corresponding system of linear equations $L_H$. 
The first step is to specify the set of all minimal linear dependencies between variables.

\begin{theorem}\label{thm:matroid}
    Let $L_H$ be the linear $H$-system defined by a 
    graph $H$ and let $F\subseteq E(H)$ be an edge subset. 
    Then $\{y_e:e\in F\}$ in the standard parametrisation of $L_H$ is a minimal linearly dependent set if and only if $F$ is an even cycle in $H$.
\end{theorem}
\begin{proof}
    It is straightforward to see that $\{y_e:e\in F\}$ is a minimal linearly dependent set of variables if $F$ is an even cycle. Let us re-index the variables in such a way that $F=[2k]$ and $y_i=x_i+x_{i+1}$, where the addition of indices is taken modulo $2k$. Then $\sum_{j=1}^{2k}(-1)^j y_j=0$ and every proper subset of $\{y_1,\dots,y_{2k}\}$ is linearly independent.

    Conversely, suppose that $\{y_e:e\in F\}$ is a minimal linearly dependent set of variables. Then $F$ must be the edge set of a subgraph, which, abusing notation slightly, we also call $F$, with minimum degree at least two. 
    In particular, $F$ contains an even cycle with edge set $F'\subseteq F$; however, $\{y_e:e\in F'\}$ is a linearly dependent set of variables, which means $F=F'$. 
\end{proof}

The \emph{cycle matroid} of a graph $H$ is the matroid defined on the edge set $E(H)$ of a graph $H$ whose independent sets are those edge subsets of $H$ which do not contain a cycle. 
In other words, every minimal dependent set, i.e., every  \emph{circuit}, is the edge set of a cycle, which for (weakly) norming graphs $H$, which are necessarily bipartite, exactly corresponds to the minimal linear dependence condition in~\Cref{thm:matroid}. Hence, for (weakly) norming graphs $H$, the map $e\mapsto y_e$ is a matroid isomorphism from the cycle matroid of $H$ to $\{y_e : e \in E(H)\}$. 
This observation now allows us to compute the rank of the linear system $L_H$. 

\begin{proposition}\label{prop:rank}
    Let $L_H$ be the linear $H$-system defined by a (weakly) norming graph $H$. Then $L_H$ has rank $|E(H)|-|V(H)|+\kappa(H)$, where $\kappa(H)$ denotes the number of connected components in $H$.
\end{proposition}
\begin{proof}
    It is a standard fact that the rank of 
    the cycle matroid of a graph $H$ is $|V(H)|-\kappa(H)$, the number of edges in a spanning forest of $H$. Therefore, the dimension of $\Sol(L_H)$, which is the rank of $\{y_e : e \in E(H)\}$, is also $|V(H)|-\kappa(H)$, so the rank of the matrix $L_H$ with columns indexed by $E(H)$ is $|E(H)|-|V(H)|+\kappa(H)$.
\end{proof}

Together, \Cref{prop:rank} and~\Cref{thm:matroid} allow us to explicitly write out the $m\times k$ linear system $L_H$ corresponding to a bipartite graph $H$ in matrix form in two steps. 
First, compute the rank $m$ of $L_H$ using~\Cref{prop:rank}. Second, assign alternating signs to the edges in each cycle of $H$. Once we find $m$ such cycles whose corresponding $\pm 1$ row vectors are linearly independent, we are done.

\begin{example}
Let $H$ be the 1-subdivision of the complete graph $K_t$. That is, we replace each edge of $K_t$ with a path of length two. By the results of~\cite{CL16}, this is a connected weakly norming graph for every $t\geq 3$, though it was shown in~\cite{LS21} that it is not norming for any $t\geq 4$. 
As $|E(H)|=t(t-1)$ and $|V(H)|=t+t(t-1)/2 = t(t+1)/2$, $L_H$ has rank $t(t-3)/2+1$. In particular, if $t=4$, this gives the $3\times 12$ system given in~\Cref{ex:K41sub}.

For another example, let $H=K_{a,b}$ as given in~\Cref{ex:Kab}. Then we have an $m\times k$ system $L_H$ with $m=(a-1)(b-1)$ and $k=ab$. In particular, we can write the $4\times 9$ matrix $L_H$ for the case $a=b=3$ as 
\begin{align*}
    \begin{pmatrix*}[r]
         1 & -1 & 1 & -1 &   &   &   &   &\\
           &   & -1 &  1 & -1 & 1 &   &   &\\
        -1 &  &  &  1 &     &   &  1 & -1 &\\
         &  &  -1 &  &     & 1  &   &  1&-1\\
    \end{pmatrix*}.
\end{align*}
Thus, the grid norms generalise~\Cref{ex:K23} in a different way to~\Cref{ex:theta}.
\end{example}

Unfortunately, the simple characterisation of minimal dependent sets given in~\Cref{thm:matroid} does not generalise in a straightforward manner to hypergraphs. A generalisation, if it exists, must at least depend on the characteristic of the underlying field $\mathbb{F}_q$.
To illustrate the dependence on the characteristic, suppose first that $\mathbb{F}_q$ is a field of characteristic $2$. 
Then the set of variables $\{y_e:e\in F\}$ for an edge subset $F\subseteq E(H)$ in the standard parametrisation of the linear $H$-system $L_H$ is linearly dependent if and only if $F$ is an even-degree subgraph. That is, the number of edges of $F$ incident to each vertex must be even. For example, let $\{x_i+y_j+z_k:i,j,k\in\{0,1\}\}$ be the standard parametrisation of the linear $H$-system $L_H$, where $H$ is the `octahedron' $3$-graph. This is the standard $3$-graph that defines the Gowers octahedral norm, while $L_H$ defines the Gowers $U^3$-norm. Then $\{x_i+y_j+z_k: i+j+k=0 \bmod 2\}$ is a minimal linearly dependent set. 
However, this is not the case over other characteristics. Indeed, if $\{x_i+y_j+z_k: i+j+k=0 \bmod 2 \}$ is a linearly dependent set of variables, then $\sum_{i+j+k\text{ even }}c_{ijk}(x_i+y_j+z_k)=0$
for some $c_{ijk}\in\{\pm 1\}$; however, one may easily check that this is impossible for characteristics other than $2$. 

Before moving on, let us show that every linear $H$-system $L$ defined by a norming graph $H$ gives a norm which is topologically equivalent to the Gowers $U^2$-norm. 

\begin{theorem}\label{thm:graph_U2}
    Let $H$ be a norming graph and let $L$ be the linear $H$-system. Then, for every $\varepsilon>0$ and any $f: G \rightarrow \mathbb{R}$ with $\vert f\vert \leq 1$, the following hold: 
    \begin{enumerate}[(i)]
        \item $\|f\|_{L}\leq \varepsilon$ implies  $\|f\|_{U^2}\leq \varepsilon^{2/e(H)}$ and
        \item $\|f\|_{U^2}\leq \varepsilon$ implies  $\|f\|_{L}\leq \varepsilon$.
    \end{enumerate}
\end{theorem}
\begin{proof}
    By~\Cref{thm:shortest}, every norming system contains a Schatten equation of length $2\ell$ for some $\ell\in [2,e(H)/2]$. 
    Thus, $\|f\|_L\leq \varepsilon$ implies $\sum_{\xi}|\ft{f}(\xi)|^{2\ell}\leq \varepsilon$ by~\Cref{cor:domination}. In particular, we must have $\vert \ft{f}(\xi)\vert^2 \leq \varepsilon^{1/\ell}$ for all $\xi$. Since the $U^2$-norm arises from the equation $x_1-x_2+x_3-x_4=0$, we have by \Cref{prop:fourier} and Plancherel's theorem that
    \begin{equation*}
        \|f\|_{U^2} = \sum_{\xi\in \ft{G}} \vert \ft{f}(\xi)\vert^4 \leq \left(\max_{\psi\in \ft{G}} \vert \ft{f}(\psi)\vert^2\right) \sum_{\xi\in \ft{G}} \vert \ft{f}(\xi)\vert^2 = \left(\max_{\psi\in \ft{G}} \vert \ft{f}(\psi)\vert^2\right) \EE_{x\in G} \vert f(x)\vert^2\leq \varepsilon^{1/\ell},
    \end{equation*}
    yielding the required conclusion.

    Conversely, let $H$ be a graph on vertex set $[d]$. Since a norming graph must have at least one edge, we may assume that $\{1,2\}$ is an edge of $H$. 
    By the definition of $L=L_H$, we then have
    \begin{align*}
    \|f\|_L^{e(H)}&=\EE_{x_1,\ldots,x_d\in G}\left[ \prod_{ij\in E(H)} f(x_i+x_j)\right]\\
    &= \left(\EE_{x_2,\ldots,x_d\in G}\Bigg[ \prod_{\substack{ij\in E(H)\\1\notin ij}} f(x_i+x_j)\Bigg]\right)\left( \EE_{x_1\in G} \Bigg[\prod_{j\in N(1)} f(x_1+x_j)\Bigg]\right).
    \end{align*}
    By the Cauchy--Schwarz inequality, this is at most
    \begin{equation*}
        \left( \EE_{x_2,\ldots,x_d\in G}\left[ \Bigg( \prod_{\substack{ij\in E(H)\\1\notin ij}} f(x_i+x_j)\Bigg)^2\right]\right)^{1/2}
        \left( \EE_{x_2,\ldots,x_d\in G} \EE_{x_1,x_1'\in G} f(x_1+x_j)f(x_1'+x_j)\right)^{1/2}.
    \end{equation*}
    Using the fact that $\vert f\vert \leq 1$ and applying Cauchy--Schwarz once more, we see that this is at most
    \begin{equation*}
         \left( \EE_{x_3,\ldots,x_d\in G} \EE_{x_1,x_1',x_2,x_2'\in G} f(x_1+x_2)f(x_1'+x_2)f(x_1+x_2')f(x_1'+x_2')\right)^{1/4}.
    \end{equation*}
    Since $(x_1+x_2,x_1'+x_2,x_1+x_2',x_1'+x_2')$ parametrises the solutions of $y_1-y_2-y_3+y_4=0$, this is precisely the $U^2$ norm of $f$.
\end{proof}

\begin{example}
\Cref{thm:graph_U2} does not generalise to hypergraphs, in that, for $r\geq 3$, different $r$-graphs may define inequivalent systems. 
\Cref{ex:subdiv_graph} shows that the line $3$-graph of a cube gives rise to the $L^8$-norm of the spectrum. This is equivalent to the $U^2$-norm, whereas the `octahedron' $3$-graph gives rise to the $U^3$-norm.
On the other hand, Hatami, Hatami and Lovett~\cite{H24} proved that $\|\cdot\|_{L}$ is equivalent to the $U^k$-norm for some $k$ whenever $L$ is a norming system, so norming hypergraphs $H$ always give rise to norms $\|\cdot\|_{L_H}$ that are equivalent to some $U^k$-norm.
\end{example}

\section{Complex-valued functions} \label{sec:complex}

Gowers norms were originally defined for complex-valued functions~\cite{G01}. 
For instance, for any function $f: G \rightarrow \mathbb{C}$,
\begin{align*}
    \|f\|_{U_2}^4 := \EE\left[f(x_1)\overline{f(x_2)}\overline{f(x_3)}f(x_4)\right],
\end{align*}
where the expectation is taken over all solutions to the equation $x_1-x_2 -x_3+x_4=0$ and, unlike the real-valued case, some of the $f$'s have been conjugated. The question we will be concerned with here is whether, for a given (real-)norming system, it is possible to assign conjugates in a similar fashion so that the system defines a norm on the space of complex-valued functions. We stress that here we are only concerned with norming systems, since weakly norming systems have an absolute value under the expectation which makes the distinction moot. 

The analogous question for graphs was already answered in the affirmative by Lee and Sidorenko \cite{LS21}, who proved that, given a (real-)norming graph, there exists a conjugation assignment
under which the graph density defines a norm for complex-valued functions. That is, the (real-)norming property is equivalent to the complex-norming property. In this section, following their technique, we prove that the same holds for linear systems. 

Let $L$ be an $m\times k$ linear system and $\alpha$ a $0/1$-valued function defined on $[k]$. 
The \emph{complex $L$-density $t_{L,\alpha}$ with respect to $\alpha$} is defined by
\begin{align*}
    t_{L,\alpha}(f) := \EE_{(x_1,\dots,x_k)\in\Sol(L)}\left[\prod_{i=1}^k f(x_i)^{\alpha(i)}\overline{f(x_i)}^{1-\alpha(i)}\right].
\end{align*}
That is, we use $\overline{f(x_i)}$ if $\alpha(i)=0$ and $f(x_i)$ otherwise for each $i\in[k]$. We say that the pair $(L,\alpha)$ is \emph{complex-norming} if $|t_{L,\alpha}(f)|^{1/k}$ defines a norm on the space of complex-valued functions on $\mathbb{F}_q^n$ for every $n$. 
With this notation, the main result of this section is as follows. 

\begin{theorem}\label{thm:complex}
    Let $L$ be an $m\times k$ norming system. Then there exists $\alpha:[k]\rightarrow\{0,1\}$ such that $(L,\alpha)$ is complex-norming. 
\end{theorem}

To prove~\Cref{thm:complex}, we need some notation. 
Let $\FF$ be the set of all complex-valued functions defined on $\mathbb{F}_q^{\mathbb{N}}$ with finite support. 
That is, each $f\in\FF$ can be seen as a complex-valued function on $\mathbb{F}_q^n$ for a large enough $n$.
Following~\cite{LS21}, we define a \emph{decoration functional} on $\FF^k$ to be a function from $\FF^k$ to $\mathbb{C}$ that satisfies the following conditions:
\begin{enumerate}[(i)]
\item\label{it:scalar} $|\tau(cf_1,\ldots,cf_k)|=|c|^k|\tau(f_1,\ldots,f_k)|$  for each $c\in \mathbb{C}$;
\item\label{it:linear} $\tau(f_1,\ldots,f_{i-1},g+ah,f_{i+1},\ldots,f_k)$\\$=
  \tau(f_1,\ldots,f_{i-1},g,f_{i+1},\ldots,f_k)  + 
   a\tau(f_1,\ldots,f_{i-1},h,f_{i+1},\ldots,f_k)$
   
   for any $f_1, \dots, f_k, g, h\in\FF$ and any \emph{real} number $a$;
\item\label{it:conjugate}
$\tau(\overline{f}_1,\ldots,\overline{f}_k) = \overline{\tau(f_1,\ldots,f_k)}$;
\item
$\tau(f_1 \otimes g_1,\ldots,f_k \otimes g_k) =
\tau(f_1,\ldots,f_k) \, \tau(g_1,\ldots,g_k)$.
\end{enumerate}

For brevity, we will sometimes write  $\tau(f)=\tau(f,\ldots,f)$. 
Following~\cite{LS21}, we set  $s_L(f_1,\ldots,f_k):=\max_{\alpha}|t_{L,\alpha}(f_1,\ldots,f_k)|$, where the maximum is taken over all $\alpha:[k]\rightarrow \{0,1\}$. Then $s_L(\cdot)$ only satisfies a weaker `subadditivity' property rather than the linearity condition~\ref{it:linear}, namely, 
\begin{enumerate}[(i')]
    \setcounter{enumi}{1}
    \item \label{it:subadd} 
    $|\tau(f_1,\ldots,f_{i-1},g+h,f_{i+1},\ldots,f_k)|$\\$\leq
  |\tau(f_1,\ldots,f_{i-1},g,f_{i+1},\ldots,f_k)|  + 
   |\tau(f_1,\ldots,f_{i-1},h,f_{i+1},\ldots,f_k)|$.
\end{enumerate}
If $\tau$ satisfies \ref{it:scalar} and \ref{it:subadd}, then we say that $\tau$ is a \emph{weak decoration functional}. In particular, $s_L(\cdot)$ is a weak decoration functional.
By adapting the proof of \Cref{Holder}, we have the following lemma. 

\begin{lemma}\label{th:Hatami_gen}
Let $\tau$ be a weak decoration functional on $\FF^k$.
Then $|\tau(\cdot)|^{1/k}$ is a seminorm on $\FF$ if,
for any $f_1,\ldots,f_k \in \FF$,
\begin{equation*}\label{eq:Hatami_gen}
  |\tau(f_1,\ldots,f_k)|^k \; \leq \prod_{i=1}^k |\tau(f_i)| \, .
\end{equation*}
The converse also holds if $\tau$ is a decoration functional.
\end{lemma}

This in turn allows us to verify the following property of $s_L(\cdot)$. Since the proof is verbatim the same as that of \cite[Lemma~5.5]{LS21}, we omit it. 

\begin{corollary}\label{cor:s_L_seminorm}
    Let $L$ be an $m\times k$ norming system. Then $s_L(\cdot)^{1/k}$ is a norm on $\FF$.
\end{corollary} 

We will also need the following easy fact~\cite[Lemma 2.2]{LS21} saying that $|\tau(\cdot)|^{1/k}$ is a seminorm if and only if $|\tau(\cdot)|$ is convex.

\begin{lemma}\label{th:convexity_gen}
Let $\tau$ be a weak decoration functional on $\FF^k$. Then $|\tau(\cdot)|^{1/k}$ is a seminorm on $\FF$ if and only if $|\tau(\cdot)|$ is convex, i.e., $|\tau(\frac{1}{2}(f+g))| \leq \frac{1}{2}(|\tau(f)|+|\tau(g)|)$ 
for all $f,g\in\FF$. 
\end{lemma}

Thus, combining~\Cref{cor:s_L_seminorm} with this lemma shows that $s_L(\cdot)$ is a convex functional. 

\begin{corollary}\label{cor:s_L_convex}
    Let $L$ be an $m\times k$ norming system. Then $s_L(\cdot)$ is convex on $\FF$.
\end{corollary}

The final ingredient we need is the following lemma~\cite[Lemma~2.3]{LS21}, which says that the convexity of $s_L(\cdot)$ implies the convexity of \emph{some} $|t_{L,\alpha}(\cdot)|$ under appropriate technical conditions.
To state the conditions, we need some more definitions. A \emph{density functional} is a functional of the form $\tau(f)=\tau(f,f,\ldots,f)$ where $\tau$ is a decoration functional (and, crucially, not a weak decoration functional). A set $\C$ in a vector space $\FF$ is called {\it algebraically open at $f\in\C$} if, for any $g\in\FF$, 
there exists $\varepsilon > 0$ such that $f+xg\in\C$ 
for each $x \in (-\varepsilon,\varepsilon)$.
We simply say that $\C$ is \emph{algebraically open} if it is algebraically open at every $f\in \C$.

\begin{lemma}\label{th:conv_min}
Let $\tau_1,\tau_2,\ldots,\tau_r$ be non-negative real-valued density functionals on $\FF$ that satisfy: 
\begin{enumerate}[(a)]
\item\label{conv_min_1}
$\tau(f) := \max_j \tau_j(f)$  is convex on $\FF$;
\item\label{conv_min_2}
for each $j\in [r]$, there is $f \in \FF$ such that 
$\max_{\ell\neq j} \tau_\ell(f) < \tau_j(f)$;
\item\label{conv_min_4}
for any non-empty algebraically open subset $\C\subseteq\FF$, 
there is $f\in\C$ such that $\tau_j(f) > 0$ 
for all $j\in[r]$. 
\end{enumerate}
Then there exists $j\in\{1,2,\ldots,r\}$ 
such that $\tau_j(\cdot)$ is convex.
\end{lemma}

With these results in hand, we are now ready to prove~\Cref{thm:complex}, which states that every norming system is also complex norming under an appropriate conjugation assignment. 

\begin{proof}[Proof of~\Cref{thm:complex}]
    Let us first give a brief sketch. By~\Cref{cor:s_L_convex}, we already know that $s_L(\cdot)$ is convex, which verifies condition~\ref{conv_min_1} of \Cref{th:conv_min} for $\tau=s_L$ and $\tau_1, \dots, \tau_r$ any collection of functionals of the form $|t_{L, \alpha}|$. Therefore, if~\ref{conv_min_2} and~\ref{conv_min_4} hold, there exists $|t_{L,\alpha}(\cdot)|$ that is convex. Finally, \Cref{th:convexity_gen} and the `converse' part of \Cref{th:Hatami_gen} show that $(L,\alpha)$ is norming, with the latter excluding the possibility that $|t_{L,\alpha}(\cdot)|^{1/k}$ defines a seminorm that is not a norm. 

    Thus, it remains to verify the technical conditions~\ref{conv_min_2} and~\ref{conv_min_4} in~\Cref{th:conv_min}.  
    Let $\mathcal{A}$ be a minimal collection of $0/1$-valued functions on $[k]$ such that $s_L(f)=\max_{\alpha\in\mathcal{A}}|t_{L,\alpha}(f)|$. Then~\ref{conv_min_2} immediately follows from minimality once we take $\tau_1,\ldots,\tau_r$ to be the collection of $|t_{L,\alpha}|$ with $\alpha\in\mathcal{A}$. 

    It therefore suffices to verify~\ref{conv_min_4}.
    For $0/1$-valued functions $\alpha_1,\ldots,\alpha_r$ on $[k]$, consider the augmented function $\alpha:[rk]\rightarrow\{0,1\}$ that maps each $j\in [rk]$ to $\alpha_q(i)$ if 
    $j=(q-1)k+i$ for $1\leq i \leq k$. 
    Let $L(r)$ be the $rm\times rk$ system obtained by taking the disjoint union of $r$ copies of the $m\times k$ system $L$, i.e., 
    \begin{align*}
         L(r)= \begin{pmatrix}
            L &   & \cdots &  &   \\
             &   L &  \cdots &  &   \\
            &&\ddots\\
             &  &  \cdots  & L  &  \\
             &  &  \cdots  &  & L \\
        \end{pmatrix}.
    \end{align*}
    Let $g$ be the constant function $1$ 
    and let $f\in\C$ for an algebraically open set $\C\subseteq\FF$. Then there exists $\varepsilon > 0$ 
such that $\C$ contains all functions of the form $f+xg$ 
with $x \in (-\varepsilon,\varepsilon)$. 
The function $P(x) := t_{L(r),\alpha}(f+xg)$ 
is a polynomial of degree $rk$ with complex coefficients, 
since the coefficient $t_{L(r),\alpha}(f+xg)$ 
of $x^{rk}$ is $t_{L(r),\alpha}(g) = 1 \neq 0$. 
Thus,~$P(x)$ has only a finite number of zeros 
on the interval $(-\varepsilon,\varepsilon)$
and, therefore, there exists some $c\in(-\varepsilon,\varepsilon)$ 
and $h=f+cg\in\G$ such that 
$\prod_{i=1}^r t_{L,\alpha_i}(h) = t_{L(r),\alpha}(h) = P(c) \neq 0$.
\end{proof}

\section{Concluding remarks} \label{sec:conc}

Some surprisingly basic questions about the structure of weakly norming systems remain open. For instance, while we have shown that every weakly norming system must include a Schatten vector in its row space, it remains open as to whether every weakly norming system has a basis consisting only of Schatten vectors. Moreover, if this is true, is it the case that there is a basis for the row space consisting of Schatten vectors all of which have the same support size? 

In \Cref{sec:Cayley}, we saw that each (weakly) norming hypergraph gives rise to a (weakly) norming linear system. But is it the case that every weakly norming linear system arises in this way? If so, then the conjecture (see, for instance,~\cite{CL16, LS21}) that every weakly norming hypergraph is associated to a finite reflection group in a specific way would also extend to weakly norming linear systems. Since Fourier techniques are available to us in the latter context, it may be that such statements are then easier to prove. 

\vspace{5mm}

\noindent\textbf{Acknowledgements.} 
DC was supported by NSF Awards DMS-2054452 and DMS-2348859 and 
SC and JL were supported by Samsung STF Grant SSTF-BA2201-02. LV was supported by Trinity College, Cambridge through a Trinity External Research Studentship and by the London Mathematical Society through an Early Career Research Fellowship.
We would like to thank Hamed Hatami for sharing his unpublished joint work with Pooya Hatami and Shachar Lovett. 

\bibliographystyle{plainurl}
\bibliography{references}

\begin{thebibliography}{10}

\bibitem{Alt23}
Daniel Altman.
\newblock On a common-extendable, non-{S}idorenko linear system.
\newblock {\em Comb. Theory}, 3(3):Paper No. 5, 12, 2023.

\bibitem{CFS10}
David Conlon, Jacob Fox, and Benny Sudakov.
\newblock An approximate version of {S}idorenko's conjecture.
\newblock {\em Geom. Funct. Anal.}, 20(6):1354--1366, 2010.
\newblock \href {https://doi.org/10.1007/s00039-010-0097-0} {\path{doi:10.1007/s00039-010-0097-0}}.

\bibitem{CHPS12}
David Conlon, Hi{\^{e}}p H{\`a}n, Yury Person, and Mathias Schacht.
\newblock Weak quasi-randomness for uniform hypergraphs.
\newblock {\em Random Structures Algorithms}, 40(1):1--38, 2012.
\newblock \href {https://doi.org/10.1002/rsa.20389} {\path{doi:10.1002/rsa.20389}}.

\bibitem{CKLL15}
David Conlon, Jeong~Han Kim, Choongbum Lee, and Joonkyung Lee.
\newblock Some advances on {S}idorenko's conjecture.
\newblock {\em J. London Math. Soc.}, 98(3):593--608, 2018.
\newblock \href {https://doi.org/10.1112/jlms.12142} {\path{doi:10.1112/jlms.12142}}.

\bibitem{CL16}
David Conlon and Joonkyung Lee.
\newblock Finite reflection groups and graph norms.
\newblock {\em Adv. Math.}, 315:130--165, 2017.
\newblock \href {https://doi.org/10.1016/j.aim.2017.05.009} {\path{doi:10.1016/j.aim.2017.05.009}}.

\bibitem{CL20}
David Conlon and Joonkyung Lee.
\newblock Sidorenko's conjecture for blow-ups.
\newblock {\em Discrete Anal.}, 2021:Paper No. 2, 13 pp., 2021.
\newblock \href {https://doi.org/10.19086/da} {\path{doi:10.19086/da}}.

\bibitem{CL23}
David Conlon and Joonkyung Lee.
\newblock Domination inequalities and dominating graphs.
\newblock {\em Math. Proc. Cam. Phil. Soc.}, pages 1--18, 2024.
\newblock \href {https://doi.org/10.1017/S0305004124000185} {\path{doi:10.1017/S0305004124000185}}.

\bibitem{FHHK23}
Yuval Filmus, Hamed Hatami, Kaave Hosseini, and Esty Kelman.
\newblock Sparse graph counting and {K}elley--{M}eka bounds for binary systems.
\newblock Preprint available at \href{https://arxiv.org/abs/2311.12248}{arXiv:2311.12248}.

\bibitem{FPZ19}
Jacob Fox, Huy~Tuan Pham, and Yufei Zhao.
\newblock Common and {S}idorenko linear equations.
\newblock {\em Q. J. Math.}, 72(4):1223–1234, 2021.
\newblock \href {https://doi.org/10.1093/qmath/haaa068} {\path{doi:10.1093/qmath/haaa068}}.

\bibitem{GHL22}
Frederik Garbe, Jan Hladk{\'y}, and Joonkyung Lee.
\newblock Two remarks on graph norms.
\newblock {\em Discrete Comput. Geom.}, 67(3):919--929, 2022.
\newblock \href {https://doi.org/10.1007/s00454-021-00280-w} {\path{doi:10.1007/s00454-021-00280-w}}.

\bibitem{G01}
William~Timothy Gowers.
\newblock {A new proof of Szemer\'edi's theorem}.
\newblock {\em Geom. Funct. Anal.}, 11(3):465--588, 2001.
\newblock \href {https://doi.org/10.1007/s00039-001-0332-9} {\path{doi:10.1007/s00039-001-0332-9}}.

\bibitem{H10}
Hamed Hatami.
\newblock Graph norms and {S}idorenko's conjecture.
\newblock {\em Israel J. Math.}, 175:125--150, 2010.
\newblock \href {https://doi.org/10.1007/s11856-010-0005-1} {\path{doi:10.1007/s11856-010-0005-1}}.

\bibitem{H24}
Hamed Hatami.
\newblock Private communication.
\newblock 2024.

\bibitem{KLM23}
Nina Kam{\v c}ev, Anita Liebenau, and Natasha Morrison.
\newblock Towards a characterization of {S}idorenko systems.
\newblock {\em Q. J. Math.}, 74(3):957--974, 2023.
\newblock \href {https://doi.org/10.1093/qmath/haad013} {\path{doi:10.1093/qmath/haad013}}.

\bibitem{LS21}
Joonkyung Lee and Alexander Sidorenko.
\newblock On graph norms for complex-valued functions.
\newblock {\em J. Lond. Math. Soc.}, 106(2):1501--1538, 2022.
\newblock \href {https://doi.org/10.1112/jlms.12604} {\path{doi:10.1112/jlms.12604}}.

\bibitem{Lov67}
L{\'a}szl{\'o} Lov\'asz.
\newblock Operations with structures.
\newblock {\em Acta Math. Acad. Sci. Hungar.}, 18:321--328, 1967.
\newblock \href {https://doi.org/10.1007/BF02280291} {\path{doi:10.1007/BF02280291}}.

\bibitem{L08}
L{\'a}szl{\'o} Lov\'{a}sz.
\newblock Graph homomorphisms: Open problems.
\newblock Unpublished manuscript, 2008.

\bibitem{L12}
L{\'a}szl{\'o} Lov{\'a}sz.
\newblock {\em Large networks and graph limits}, volume~60 of {\em Amer. Math. Soc. Colloq. Publ.}
\newblock American Mathematical Society, 2012.
\newblock \href {https://doi.org/10.1090/coll/060} {\path{doi:10.1090/coll/060}}.

\bibitem{Sid20}
Alexander Sidorenko.
\newblock Weakly norming graphs are edge-transitive.
\newblock {\em Combinatorica}, 40(4):601--604, 2020.
\newblock \href {https://doi.org/10.1007/s00493-020-4468-3} {\path{doi:10.1007/s00493-020-4468-3}}.

\bibitem{SkokanThoma04}
Jozef Skokan and Lubos Thoma.
\newblock Bipartite subgraphs and quasi-randomness.
\newblock {\em Graphs Combin.}, 20(2):255--262, 2004.
\newblock \href {https://doi.org/10.1007/s00373-004-0556-1} {\path{doi:10.1007/s00373-004-0556-1}}.

\bibitem{Sz15}
Bal{\'a}zs Szegedy.
\newblock An information theoretic approach to {S}idorenko's conjecture.
\newblock Preprint available at \href{https://arxiv.org/abs/1406.6738}{arXiv:1406.6738}.

\end{thebibliography}

\end{document}